\documentclass[12pt, reqno]{amsart}
\usepackage{amssymb}
\usepackage{amsmath}
\usepackage{amscd}
 \usepackage{tikz-cd}
\usepackage{float}
\usepackage{graphicx}
\usepackage{amsfonts}
\usepackage{ulem}
\usepackage{cancel}
\usepackage{hyperref}
\newtheorem{theorem}{Theorem}
\newtheorem{corollary}{Corollary}
\newtheorem{lemma}{Lemma}
\newtheorem{proposition}{Proposition}
\newtheorem{definition}{Definition}

\newtheorem{example}{Example}
\newtheorem*{note}{\bf Notational convention}

 \newcommand{\ldbrack}{{[}} \newcommand{\rdbrack}{{]}}
\numberwithin{theorem}{section}
\numberwithin{definition}{section}
\numberwithin{example}{section}
\numberwithin{lemma}{section}
\numberwithin{corollary}{section}
\numberwithin{proposition}{section}

\theoremstyle{definition}
\newtheorem{remark}{Remark}
\numberwithin{remark}{section}

\newcommand{\A}{\mathcal{A}}

\newcommand{\s}{\mathsf{s}}
\newcommand{\T}{\mathsf{t}}
 \newcommand{\OO}{\mathsf{o}}

\newcommand{\aaa}{\mathsf{a}}

\newcommand{\CC}{\mathbb{C}}
\newcommand{\FF}{\mathbb{F}}
\newcommand{\KK}{\mathbb{K}}

\newcommand{\RR}{R'}
\newcommand{\ZZ}{\mathbb{Z}}
\newcommand{\TL}{\mathrm{TL}}
\newcommand{\BMW}{\mathrm{BMW}}

\numberwithin{equation}{section}

\begin{document}

\title{Framization and deframization}

\author{Francesca  Aicardi}
 \address{ Sistiana 56, 34011 Trieste (Italy)}
\email{francescaicardi22@gmail.com}
 \author{Jes\'us Juyumaya}
 \address{Instituto de Matem\'{a}ticas, Universidad de Valpara\'{i}so,
 Gran Breta\~{n}a 1111, Valpara\'{i}so 2340000, Chile.}
 \email{juyumaya@gmail.com}
\author{Paolo Papi}
\address{Dipartimento di Matematica G. Castelnuovo, Sapienza Universit\`a di Roma, P.le Aldo Moro 5, 00185 Roma, Italy}
 \email{papi@mat.uniroma1.it}

\date{}
\thanks{Corresponding author: Paolo Papi}

\begin{abstract}
Starting from the geometric construction of the framed braid group,
we define and study the framization of several Brauer-type monoids and also
the  set partition  monoid, all of which appear in knot theory. We introduce the
 concept of deframization,    which is  a  procedure to obtain a tied monoid from a
given framed monoid. Furthermore, we show in detail how this procedure works on
the monoids mentioned above. We also discuss the framization and
deframization of some algebras, which are deformations, respectively,  of the framized  and deframized  monoids
 discussed here.
\end{abstract}
\maketitle
\centerline{{\sl To Peter Littelmann on the occasion of his retirement}}
\setcounter{tocdepth}{1}
\tableofcontents
\section{Introduction}

The framed braid group $\mathcal{F}_n$ appears in  the study of so-called templates, see \cite{MeTuPRA1991}. The definition of $\mathcal{F}_n$ can be given in  geometric terms by  attaching an integer, called {\it framing}, to each strand of a geometric braid. Equivalently,
a framed braid can be represented as a ribbon braid, that is obtained by thickening a strand  into  a ribbon, and where the framing is the number of half-twists of the ribbon. For the role of $\mathcal{F}_n$ in knot theory, see \cite{KoSmPAMS1992, JuLaAinM2013} for instance.  From an algebraic point of view it turns out that $\mathcal{F}_n=\ZZ^n\rtimes B_n$, where $B_n$ denotes the group of braids on  $n$ strands, see \cite{KoSmPAMS1992}. The $d$-modular framed braid group $\mathcal{F}_{d,n}$ is formed by braids having as framing an integer modulo $d$,   and we have that $\mathcal{F}_{d,n}$ is the group  $(\ZZ/d\ZZ)^n\rtimes B_n$.
Now, just as $B_n$ has the symmetric group $S_n$ as a natural quotient, $\mathcal{F}_n$ maps onto the $d$-modular framed symmetric group $S_{d, n } = (\ZZ/d\ZZ)^n\rtimes S_n$. All of these groups are closely related through  knot algebras, namely, the Hecke and the Yokonuma-Hecke algebras, \cite{YoCRASP1967, JuJKTR2004, MaIMRN2017}. Indeed, the Hecke algebra can be realized as a quotient of the group algebra of $B_n$ or as a deformation of the group algebra of $S_n$, and the Yokonuma-Hecke (Y-H for short) algebra can be realized as a quotient of the group algebra of $\mathcal{F}_n$ or
 as a deformation of the group algebra of $S_{d, n }$.  Namely, a presentation of the Y-H algebra is obtained by adding framing generators to the defining generators of the Hecke algebra with commutation relations coming from  $ S_{d,n}$ and a certain deformation of the involutive generators of $S_n$, see \cite{JuJKTR2004} for details.
  This  deformation  involves  necessarily  idempotent elements  $\bar e_i$ (formula (\ref{ei})), expressed as  averages of  products  of  the framing  generators,  and maps to  the corresponding relation of the Hecke algebra when the framing generators are  sent to 1. Observe also  that the dimension of the Y-H algebra   equals  the  cardinality   of $ S_{d,n}$.  We thus   say  that  the Y-H algebra
 is a framization of the Hecke algebra as well as   $S_{d,n}$ is a framization of $S_n$.

  In this paper   we introduce    the notion of {\it deframization},   the prototypical  example  being  the so-called  bt-algebra, which is considered as  deframization of the Y-H algebra.

The bt-algebra  was  defined in \cite{AiJu}  and    it is  in fact obtained from the Y-H algebra,  by  eliminating the  framing generators  and  adding  new  independent   generators $e_i$, that  replace the idempotent  $\bar e_i$ in the defining relations of the Y-H algebra,  and  that satisfy   the  same     relations as  the $\bar e_i$ with the  other generators    in the Y-H algebra.
The  quite  unexpected fact  is that  the  $e_i$,  called  {\it  ties}  for their  diagrammatical  interpretation,  define  a  partition  of the  set  of  strands  of  the  braid  associated to  an element of the bt-algebra. In fact, it was  proven in \cite{SRH}, \cite{EsRyJPAA2018} that  the dimension of the  bt-algebra  is the  cardinality of the  ramified  symmetric monoid  $P_n \rtimes S_n$,  $P_n$  being the    monoid of  set partitions  of $[n]$.  Note that  this ramified monoid  has  a presentation involving   the transpositions  and the  same  ties $e_i$, so that it is also called  tied symmetric monoid  and  the bt-algebra can be considered as  a  deformation  of  it.

The  bt-algebra   is  the  algebraic counterpart of tied  links  (as  the Hecke algebra is  for  classical  links), and  a  trace on it  allows us to  recover a  Homflypt type invariant for tied links
 $\widetilde{\mathcal X}$ that is more powerful than the Homflypt polynomial on links. A natural question then arises: is there a Kauffman type invariant for tied links analogous to $\widetilde{\mathcal X}$? The answer is affirmative: this invariant, denoted by $\widetilde{\mathcal L}$, was built in \cite{AiJuMathZ2018}; in the same paper the tBMW algebra is defined, which is related to $\widetilde{\mathcal L}$ as the bt-algebra is related to $\widetilde{\mathcal X}$.

However, the definition of the tBMW algebra  does not use any intermediate algebra.   In fact  it arises  from diagrammatic considerations on the elementary tangles and braids that define the BMW algebra \cite{BiWeTAMS1989}, together with the ties generators and the so-called tied tangles, which are  elementary  tangles with a tie connecting their cups.

Recall  also  that     the  ties,  together with the {\it tied  generators}, have  been  succesfully  introduced to give a presentation for  the  so called  ramified  monoid (Jones,  Brauer, Inverse  symmetric)\cite{AiArJuJPAA2023},   in which  the  ties  define  a  partition in the  set of  strands, generalizing the  case of the tied  symmetric  monoid.

  This  paper  is  devoted to  answer the following natural  question: can   the  monoids quoted above,  and  possibly their related  algebras,   be  recovered  from the  corresponding  framed  monoid or  algebras?  We   underline  the fact  that  our  approach  is  finalized to  applications to  knot  theory, therefore   the  geometric/topological   meaning  related  to the diagrammatic  interpretation    is  our   guideline.

 We start   by defining   the framization of some monoids. Then, we  introduce  the notion  of deframization, a procedure allowing to   build, starting from  a  framed monoid  algebra,
  a new monoid,  which can be understood as a ramified monoid  in the sense of \cite{AiArJuJPAA2023}.

The  main tool  of  the  deframization procedure  is the  definition, in the  framed  monoid  algebra,  of  idempotent  elements  that  generalize  the  $\bar e_i$  in the  Y-H  algebra  and  in the  framed  symmetric  monoid. Such an element,  expressed  as  a suitable average  of  products of  framings  with  a generator  of the monoid,  has  the elementary property  of  allowing  a   single framing to  commute  with that  generator,  and, in term of  diagrams, to move  from  a  strand  to another  one.  For this  reason these  elements  are  called {\it  bridges}.  The  remarkable  result is that,  defining  new independent  generators  that keep  the  same  commuting relations  as the bridges  among them  and with the  other generators  of the  framed monoid, they  behave as  tied  generators  and,  by  removing the  framings, they  generate  the corresponding  ramified  monoid.

 At level of  monoid algebra,  the definition  of the bridge  elements is always possible.  The same is not true  in general  for algebras in which  relations  proper of  the  deformation, i.e., not shared  with the  algebra  monoid,  make the bridges not independent  from  one another.    This  leads us  to  define the  framization and  deframization of  an algebra, which is a deformation of a  monoid  algebra,   only in the  case  when   the  framized  and  deframized algebras  are  deformations of  the  corresponding  framed  and  tied  monoid algebras, and  keep  their  dimension.

The paper is organized as follows. Section 2 is dedicated to giving definitions and notational conventions to be used in a paper on {\it knot structures}; in particular, we introduce  the monoid of set partitions, Jones and Brauer monoids,  the Hecke and BMW algebras among others. Section 3 discusses the framization of some  knot monoids. The definition of these  framed monoids is  inspired by   the  classic framed  braid group, but the framing is now an integer modulo $d$. Now, a significant interpretation of the framings modulo $d$ is as beads  on the thread that are counted modulo $d$. With
 this interpretation, a framed monoid with framings modulo $d$ is called an {\it abacus monoid}.  We introduce in Subsection 3.3,  by generators and relations, the framed Jones  and framed Brauer monoids (see Subsection 3.4) and two version of framed rook monoid, see Subsection 3.5. We show that they are isomorphic to their abacus version and, as a byproduct, we  calculate their cardinality: see Theorems \ref{propJdn}, \ref{Brauerabacus}, \ref{Rook1abacus}, \ref{Rookabacus}.

 Section 4 begins by  introducing  the notion of   deframization, which is given in   Definition \ref{deframizationdef}. In short,  the deframization of a framized monoid $M_{d,n}$ is a procedure that builds from a monoid algebra $A$ of $M_{d,n}$ a certain  quotient algebra $D(A)$, that turns out  to be the monoid algebra of the ramified (tied) monoid of $M$.  Then it is shown how the procedure works for the framizations of the rook monoids, the  Jones and Brauer monoids.

 In Section 5 we try to  extend the notion of   framization and  deframization  to algebras which are deformations of the monoid algebras considered in this paper. Thus, we prove that, besides    the bt-algebra,  the Temperley-Lieb algebra admits a  tied  version  with  two parameters   that  can be recovered  by    deframization   of  a  framed  algebra  with  $d$  parameters.
The   BMW algebra, which is  a deformation of  the  Brauer  monoid  algebra,  does not admit  a  framization but  a tied  version,  which is in fact  a deformation of the tied  Brauer monoid algebra obtained  by  deframization.

\section{Background}
 In this section we introduce some notations used along the paper and we recall the main objects we will deal with in our work: braid and symmetric  groups,  rook monoid, set partition  monoid, Brauer and  Jones  monoids,  with  their  presentations  by  generators  and  relations.  We refer to these monoids as knot monoids since they play a role in the construction of knot invariants. They are related, respectively, to  the following knot algebras: Hecke algebra, rook algebra, Temperley-Lieb algebra, BMW algebra and bt-algebra.

\subsection{Notational convention.}
Whenever the range of  indices of the generators is  not specified, we stipulate  that the indices take any value for which the respective generator is  defined.

By  $\ldbrack r,s\rdbrack$  we  will denote  the  set of integers i such that  $r\le i \le s$  and by $[n]$ the set  of  integers $i$  such that  $1\le i \le n$.

The terminology {\it $K$-algebra} means unital associative algebra over the field $K$.

 For any monoid $M$ we denote by $K[M]$ the monoid algebra of $M$.
 Recall that, if $ \langle X, R\rangle $ is a presentation of $M$, then the algebra $K[M]$ can be presented by generators $X$ and relations $R$.

\subsection{Braid and symmetric groups}
Let $B_n$ denote the $n$ strand  braid group, that is,  the group with generators $\sigma_1,\ldots ,\sigma_{n-1}$ satisfying the  {\it braids relations}
\begin{equation} \label{braid}\sigma_i\sigma_j=
\sigma_j\sigma_i\text{ if $|i-j|>1$ and} \quad \sigma_i\sigma_j\sigma_i =
\sigma_j\sigma_i\sigma_j\text{ if $|i-j|=1$}.\end{equation}
 We denote by $S_n$ the symmetric group on $n$ letters and by $s_i$ the elementary transposition $(i, i+1)$. Recall that the $s_i$ define the  Coxeter  presentation of $S_n$, whose relations are:
\begin{align}\label{Coxeter-Moore1}
 &s_i^2=1,\\
&s_is_j= s_js_i,\quad \text{if $|i-j|>1$}, \quad  s_is_js_i =s_js_is_j,\quad \text{  if $|i-j|=1$.}\label{Coxeter-Moore2}
\end{align}

\subsection{The rook monoid $R_n$}\label{RM} The rook monoid $R_n$ is a generalization of the symmetric group and is also denoted $IS_n$, for inverse symmetric monoid.  For more details see  \cite{LiMSM-AMS1996} for instance.  $R_n$ is  presented by generators  $s_i,\ldots ,s_{n-1}, p_1,\ldots , p_n$ subject to the Coxeter relations among the $s_i$  together with the following relations:
\begin{align}
p_i^2=p_i,&\qquad  p_ip_j =p_jp_i,\label{Mrook1}\\
p_is_j = s_jp_i \quad \text{for $i< j$,}&\qquad p_is_j= s_jp_i=  p_i\quad \text{for $1\leq j <i\leq  n,$}\label{Mrook3}\\
p_is_i p_i= p_{i+1}&\qquad \text{for $i\in \ldbrack 1,n-1\rdbrack.$}\label{Mrook4}
\end{align}

\begin{remark}\label{Rookgenerators}     We  will  use as well  another  presentation having the     $r_i$ instead of  the $p_i$ as generators,  see    \cite{AiArJuMMJ}.  In this   presentation,   relations (\ref{Mrook1})--(\ref{Mrook4})  are  replaced by
\begin{align}
r_i^2&=r_i,\quad  r_ir_j =r_jr_i,\label{M2rook1}\\
r_js_i &= s_ir_j    \qquad   j\not=i,i+1, \label{M2rook2}\\
r_is_i &= s_i r_{i+1},\quad r_{i+1}s_i= s_i r_i  \qquad \text{for $i\in \ldbrack 1,n-1\rdbrack
.$}\label{M2rook3}\\
r_is_i r_i&  = r_i r_{i+1}\qquad \text{for $i\in \ldbrack 1,n-1\rdbrack.$}\label{M2rook4}
\end{align}
The $p_i$ and  $r_i$  are  related by
 the  equations:
\begin{align}  &  r_i=   s_{i-1} s_{i-2}\cdots s_1 p_1 s_1 s_2 \cdots  s_{i-1}, \label{RHO1}\\
               & p_i= r_1 r_2 \cdots r_i. \label{RHO2}
\end{align}
\end{remark}

\subsection{The monoid  $P_n$  of  set partitions}

The monoid $P_n$ is formed by the set  partitions of  $\ldbrack 1,n \rdbrack$ with  product given by  refinement, see \cite{AiArJuJPAA2023} for details.
 We recall  that  (\cite[Theorem 2]{FiBAMS2003})
the monoid $P_n$ can be presented by generators $e_{i,j}$, with $i,j \in\ldbrack 1, n \rdbrack$ and $i<j$, satisfying
 the following relations:
\begin{align}
e_{i,j}^2&=e_{i,j}\quad\text{for all }i<j,\label{Pn1}\\
e_{i,j}e_{r,s}&=e_{r,s}e_{i,j}\quad\text{for all }i<j\text{ and }r<s,\label{Pn2}\\
e_{i,j}e_{i,k}&=e_{i,j}e_{j,k}=e_{i,k}e_{j,k}\quad\text{for all }i<j<k.\label{Pn3}
\end{align}

\subsection{Jones and Brauer monoids}\label{subsectionJB}
We denote by $J_n$ the Jones monoid (\cite{BDS, 17}), that is the monoid generated by $t_1,\dots,t_{n-1}$ subject to the relations:
\begin{align}
& t_i^2  =t_i,   \label{J01} \\
& t_it_j   =t_jt_i \quad \text{if $\vert i-j\vert>1$, and }
\quad t_it_jt_i   =t_i \quad \text{if $\vert i-j\vert=1$.} \label{J02}
\end{align}
 Recall that the cardinality of $J_n$ is the $n$-Catalan number $Cat_n:= \frac{1}{n+1}\binom{2n}{n}$.

The Brauer monoid, denoted $Br_n$,  is defined by the generators $s_1,\dots,s_{n-1}$, $t_1,\dots,t_{n-1}$ satisfying
relations (\ref{Coxeter-Moore1})-(\ref{Coxeter-Moore2}) and  (\ref{J01})-(\ref{J02}), together with  the following {\it mixed} relations.
\begin{align}
&  t_i s_i=  s_i t_i = t_i, \label{B01}\\
& t_is_j   =s_jt_i \quad \text{if $\vert i-j\vert>1$},\label{B02}\\
&  s_it_jt_i =s_jt_i, \quad t_it_js_i= t_is_j\quad \text{if $\vert i-j\vert=1$}.\label{B03}
\end{align}
These relations imply
\begin{equation}\label{B04}
s_it_js_i   =s_j t_i s_j, \quad
 t_i s_j t_i =  t_i  \quad \text{and}\quad
 t_i t_j = t_i s_j s_i= s_js_it_j,  \quad \text{for $\vert i-j\vert=1$.}
\end{equation}
Recall that the cardinality of $Br_n$ is $(2n-1)!!$.  See \cite{KuMaCEJM2006} for details.
\subsection{The Hecke  algebra} For an indeterminate $v$ over $\CC$,  we set  $\KK = \CC(v)$.
 Denote by $\mathrm{H}_n(v)$ the Hecke algebra, that is the  $\KK$-algebra generated by $h_1,\ldots ,h_{n-1}$ subject to the braid relations among the $h_i$ and the quadratic relations $h_i^2= 1  + (v-v^{-1})h_i$,  $i\in\ldbrack 1,n-1\rdbrack$. Recall that $\mathrm{H}_n(v)$ can be regarded as a deformation of the group algebra of the symmetric group $S_n$, i.e. $\KK[S_n]= \mathrm{H}_n(1)$, and also  as the quotient of $\KK[B_n]$ over the two-sided ideal generated
by the quadratic expressions $\sigma_i^2-1 -(v-v^{-1})\sigma_i,    i\in  \ldbrack 1,n-1\rdbrack$.

\subsection{The rook algebra}\label{rA}  The rook algebra, denoted by $\mathcal{R}_n(v)$, is defined by generators $T_1,\ldots , T_{n-1}$, $P_1, \ldots , P_n$ satisfying  the Hecke relations among the $T_i$ together with the following relations:
\begin{align}
P_i^2  &=P_i,\quad \label{Arook1}\\
P_iT_j = T_jP_i \quad \text{for $i< j$,}&\qquad P_iT_j  = T_j P_i =   vP_i\qquad \text{for $1\leq j <i\leq  n,$}\label{Arook3}\\
P_{i+1}  = v  P_iT_i^{-1}P_i&\qquad \text{for $i\in\ldbrack 1, n-1\rdbrack$.}\label{Arook4}
\end{align}

The rook algebra was introduced by  L. Solomon in   \cite{SoJA2004}.
The presentation of $\mathcal{R}_n(v)$ described above  is due to T. Halverson, see \cite[Section 2]{HaJA2004}.

\subsection{ The Temperley-Lieb  algebra}\label{TL}  Let  $n$ be a  positive integer, $u$ an indeterminate and   $\KK = \CC(u)$. We denote by $\mathrm{TL}_{n}=\mathrm{TL}_{n}(u)$ the Temperley-Lieb algebra, that is the  $\KK$-algebra generated by $t_1, t_2,\ldots, t_{n-1}$,  satisfying the following relations
\begin{equation}
t_i^2=u t_i \quad \text{for all $i$},  \quad  t_it_j   =t_jt_i \quad \text{if $\vert i-j\vert>1$},
\quad t_it_jt_i   =t_i \quad \text{if $\vert i-j\vert=1$}. \label{Z0}\\
\end{equation}
 Note that   $\mathrm{TL}_n(u)$ is  a  deformation of  the  monoid  algebra $\mathbb C[J_n]$, i.e.,
$\mathrm{TL}_n(1)=\mathbb C[J_n]$.

\subsection{The BMW  algebra}\label{BMWalgebra}

 We define the BMW algebra (\cite{BiWeTAMS1989, MuOJM1987}) as  in  \cite[Definition 1]{CoEtAlJA2005}.

Let $a$, $q$ be two indeterminates over $\CC$.  The  algebra   $\mathrm{BMW}_n$,  is the $\CC(a,q)$-algebra   defined by braid generators $g_1, \ldots ,g_{n-1}$,  invertible,  subject to  braid relations,  and tangle generators $t_1, \ldots ,t_{n-1}$,  subject  to  relations  (\ref{J02}),  and  to the  mixed  relations
\begin{align}
 & g_i t_i   =  a^{-1} t_i, \label{BMW1}\\
&  t_ig_j^{\pm 1}t_i  =  a^{\pm 1} t_i , \qquad \text{for}\quad \vert i-j\vert =1,\label{BMW2} \\
&   g_i -g_i^{-1}  =  (q-q^{-1})(1 - t_i). \qquad \label{BMW3}
\end{align}

\begin{remark}\label{paramx}
    It is well-known that   $\mathrm{BMW}_n$ is a deformation of the
    Brauer algebra, see \cite{WeAM1988}.  Observe that the deformation of  (\ref{J01}) is a consequence of (\ref{BMW3}) and (\ref{BMW1}). More precisely we have
    \begin{equation} t_i^2= x t_i, \quad \text{ with}\quad  x=  \frac{a-a^{-1}}{q-q^{-1}} + 1, \label{parx} \end{equation}
   so that, in the above presentation,  $\BMW_n=\CC[Br_n]$ when $a=q=1$.
\end{remark}

\section{Framization}

Recall that diagrammatically  the $d$-framed, or  framed, braid group  is constructed  by attaching an integer $\mod d$ to each strand of $B_n$. Following this construction we introduce the framed monoid for the   knot monoids mentioned above. To do this  we need  the following preparation.
 For every positive integer  $d$ we denote $\ZZ/d\ZZ =\{1,2,\ldots , d\}$ the group of integers modulo $d$, and  by  $C_d$ the group $\ZZ/d\ZZ$ written in multiplicative notation, so that
$C_d$ has the presentation  $ \langle z \,;\, z^d=1\rangle$. We  define  also $C_{d,n}$ as the group presented by generators $z_1,\ldots,z_n$ and relations
\begin{equation}\label{PreCdn}
 z_i^d=1,\quad z_iz_j =z_jz_i.
\end{equation}

 The elements of  $C_{d,n}$    can  be  diagrammatically  represented    by  $n$  vertical  parallel  lines,  to each  one  of  which is  attached  an  integer.  The  product is  obtained  by  concatenation  and  the  integers  are  added  modulo $d$.  An  equivalent way  to  represent  the  integer  $k$  attached  to  a  line  is  to  put  $k$  beads  on the line.  These  lines   will  be interpreted   as strands   in the     contexts of  braid  type.   For this  reason we refer to  $C_{d,n}$  as   {\it abacus}  monoid.

\begin{figure}[H]
 \includegraphics[scale=0.9]{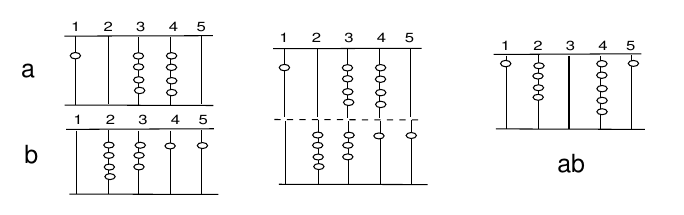}
 \caption{  Two elements of $C_{7,5}$ and their  product. }\label{Fig00}
 \end{figure}

 The  abacus  monoid  is  manifestly  generated by  the  special  elements  $\OO_i$,  $i\in \ldbrack 1,n \rdbrack$, consisting of  diagrams differing from the identity by a sole  bead on the  line $i$,  see  Figure \ref{Fig04}.
Observe that $ \OO_i$ is invertible with inverse $ \OO_i^{d-1}$.
\begin{figure}[H]
\includegraphics[scale=0.8]{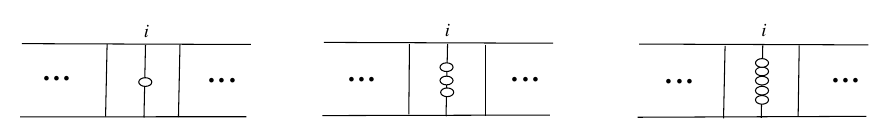}
\caption{Diagrams of the  generators  $\OO_i$, $\OO_i^3$ and  $\OO_i^{-1}= \OO_i^{5}$ in $C_{6,n}$.}\label{Fig04}
\end{figure}

\subsection{Framization of symmetric   and  braid  group}

 We denote by   $\mathcal{S}_{n}$ the group $\ZZ\wr S_n = \ZZ^n\rtimes S_n$  derived from the permutation action of $S_n$ on $\ZZ^n$. This action together with the natural epimorphism of $B_n\to S_n$ produces the so called {\it framed braid} group
 $ \mathcal{F}_n:= \ZZ  \wr  B_n$,  which appears firstly in \cite{KoSmPAMS1992, MeTuPRA1991}.  In these papers the elements of  $\mathcal{F}_n$  are the usual geometrical braids but with an integer attached to each strand.

 Define
\begin{equation}\label{CF}S_{d,n}=C_d  \wr  S_n ,\quad\mathcal{F}_{d,n}=C_d\wr B_n.\end{equation}

We  say  that  $S_{d,n}$  is  the  result of framization of  the  symmetric   group,  and    $\mathcal{F}_{d,n}$  is  the  result of  the  framization of  the  braid  group.

\begin{remark}\label{RemPre} Observe that $S_{d,n}$  has a presentation with generators $z_i$ and  $s_i$  satisfying  the relations in  (\ref{PreCdn}),  the Coxeter relations (\ref{Coxeter-Moore1}),  (\ref{Coxeter-Moore2}) among the $s_i$ , together with the relation
\begin{equation} z_j  s_i =  s_i z_{s_i(j)},\label{sizi}\end{equation}  where $s_i(j)$ denotes the action of the transposition $(i,\, i+1)$ on $j$.
 \end{remark}

Note that  $\mathcal{F}_{d,n}$ is presented by $z_i$ and  $\sigma_i$ satisfying  the relations in  (\ref{PreCdn}),  the braid relations  (\ref{braid}) among the $\sigma_i$, together with the relation $z_j \sigma_i =   \sigma_i z_{s_i(j)}$.
Also, it is  evident  that  relations  between  $z_i$  and  $s_j$  or $\sigma_j$   are  compatible  with  the diagrammatic  interpretation of  the  framing as  a  bead  sliding  along  the  strand, so that  $\mathcal{F}_{d,n}$ can be viewed as an abacus monid. The next figure show  an example of the product in  $\mathcal{F}_{3,5}$  as  abacus monoid.

 \begin{figure}[H]
 \includegraphics[scale=0.9]{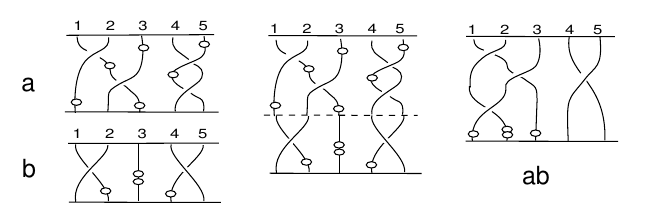}
 \caption{  Two elements of $\mathcal{F}_{3,5}$ and their  product. }\label{Fig00a}
 \end{figure}


\subsection{The framed    set partition monoid}
 The  first  framed  monoid  we  introduce  here is  obtained  by  the framization of  the  set partition  monoid  $P_n$.

The elements of $P_n$ can be represented diagrammatically by $n$  parallel vertical lines,   corresponding to the elements, and horizontal wave lines, called {\it  ties}, each one connecting a   pair of lines. If two  lines   are  connected  by a tie,  the  corresponding  elements  belong to  the  same  block of  the partition. In diagrammatic terms, the element $e_{i,j}$    corresponds to a single  tie connecting the lines $i$ and $j$.
 The  diagrammatic  representation is  standard  if  the ties  of a  block  $\{a_1< a_2 <\ldots< a_m\}$  are  exactly  the  $m-1$  ties connecting  $a_i$ to  $a_{i+1}$.  The  standard  representation is  unique. For instance,  Figure \ref{Fig000a} represents the set partition $\{\{1,2,4\},\{3,5\} \}$   as product of partitions   $\{\{1,4\},\{3,5\} \} $ and $\{\{2,4\}  \} $.

 \begin{figure}[H]
 \includegraphics[scale=0.9]{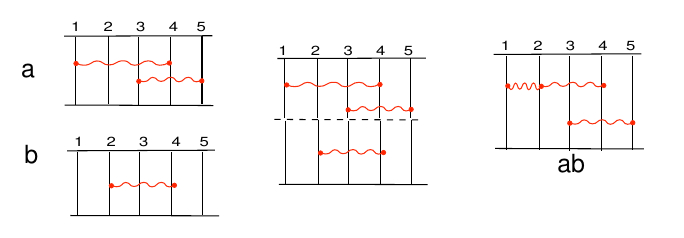}
 \caption{  Two elements of $P_5$ and their  product.}\label{Fig000a}
 \end{figure}
\begin{definition} \label{Pdn} The $d$-framed set partition monoid $P_{d,n}$ coincides with  $P_n$  for $d=1$, and for $d>1$ is the monoid obtained by putting  together the presentation of  $P_n$ by  the  $e_{i,j}$  subject to   (\ref{Pn1})--(\ref{Pn3})
 with that of $C_{d,n}$ given in (\ref{PreCdn}), and adding the following relations:
\begin{equation}\label{PrePdn}
z_ke_{i,j} = e_{i,j}z_k,\quad z_ie_{i,j} = z_je_{i,j}.
\end{equation}
\end{definition}

Here again we can consider the  abacus version of this monoid. Thus equation  (\ref{PrePdn})
  says that the tie  $e_{i,j}$   allows  the  beads    to  move  freely  from strand $i$ to  strands  $j$  and  vice-versa, see  Figure \ref{Fig0000a}.

\begin{figure}[H]
 \includegraphics[scale=0.9]{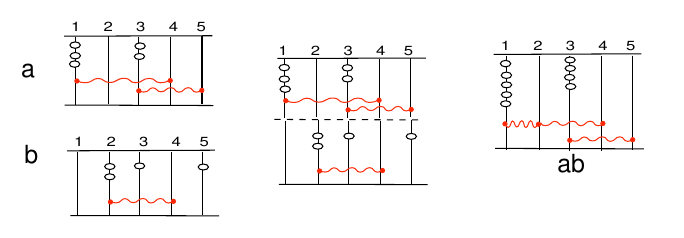}
 \caption{Two elements of $P_{7,5}$ and their  product.}\label{Fig0000a}
 \end{figure}

\subsection{Framed Jones monoid}  Before introducing the framed Jones monoid, we need to remember some facts about the Jones monoid $J_n$ and  its  diagrammatical  interpretation. 
Namely, we will recall its diagrammatic realization and the normal form of its elements. 
More precisely,  consider the monoid formed by  diagrams  made by $n$  noncrossing  arcs,  with endpoins on the integer points  $1,\dots, n$ of  two parallel  lines.  The product is  defined by concatenation, in which the loops are neglected.  See  Figure \ref{Fig01}.
 In \cite{BDS}, \cite{East2021} it is proved that  a presentation of this monoid  is given by relations \eqref{J01} and \eqref{J02}.
 \begin{figure}[H]
 \includegraphics[scale=0.8]{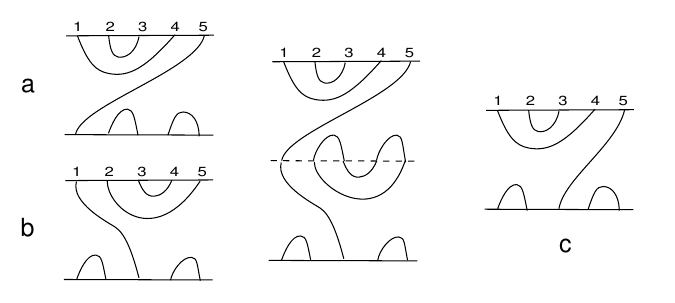}
 \caption{  Two elements of $J_5$ and their  product.}\label{Fig01}
 \end{figure}
\vskip10pt

{\bf Notation.} 
\begin{itemize}\item
In diagrammatic terms, the generator $t_i$ is denoted by $\T_i$ and is the diagram with only the pair of brackets with endpoints $i$ and $i+1$, it is called {\it tangle} (cf. Figure \ref{Fig02}a).
If it is not necessary we will not distinguish between $t_i$ and $\T_i$.


\item 

  The two horizontal parallel segments where the  endpoints of the arcs  of a  diagram lie, are  said  respectively the {\it top}  and  the {\it bottom segment}.
 An arc  with  both  endpoints on  the top (bottom)  segment  is called  {\it up-bracket} ({\it down-bracket}).  Otherwise it is  called  {\it line}. Moreover,    we  call  {\it ne-line} ({\it nw-line}) a  line  whose  lower endpoint  is smaller (greater)  than  the upper one. For instance,  the diagrams (a) and (c) in Figure \ref{Fig01}  have a sole ne-line,  and the diagram   (b) has a sole nw-line.
\end{itemize}
\vskip5pt

Recall that every element in $J_n$ can be written in a normal form, see  \cite[Aside 4.1.4]{JoIM1983}. More precisely, for  $i\geq j$, we set $
 U_{ij}=t_it_{i-1}\cdots t_j$.
 Thus every element $w\in J_n$ can be written in a normal form $T^{+}(w)$, where
\begin{equation}\label{mon1}
T^{+}(w) = U_{i_1j_1}\cdots U_{i_k j_k},\quad i_1<i_2\ldots<i_k,\,
j_1<j_2\ldots<j_k.
\end{equation}
The dual  $T^{-}(w)$ of $T^{+}(w)$ is    another    normal form of  $w$ given by
 \begin{equation}\label{mon2}
 T^{-}(w)= V_{i_1j_1}\cdots V_{i_h j_h},\quad i_1>i_2\ldots>i_h,\,
j_1>j_2\ldots>j_h,
\end{equation}
where, for  $i\leq j$, $V_{ij}:=t_it_{i+1}\cdots t_j$.

\begin{definition}\label{Not2}\rm We say that   an integer $g\in\ldbrack 1,n-1\rdbrack$ is a {\it  gap} of $w\in J_n$, if the  indices  $g$  and  $g-1$ do not appear in the normal  form of $w$.  Note that   the notion of gap   is  independent  from the  normal  form we use. We  denote  by $\underline j_1,\dots,\underline j_k$  the  second indices   in the $U_{ij}$  of  $T^{+}(w)$ and  $\overline i_1,\dots,\overline i_h$    the first indices  in the $V_{ij}$  of  $T^{-}(w)$.
\end{definition}

\begin{example}\rm  The diagram in Figure \ref{Fig01}a  is  represented  by the  normal forms
\begin{equation}\label{normalform} T^+(w) =U_{2,1} U_{3,2}U_{4,4}= t_2  t_1  \ t_3 t_2 \ t_4= t_2 t_3  t_4  \  t_1 t_2= V_{2,4} V_{1,2}=T^-(w).  \end{equation}
See  Figure \ref{Fig02}b.

\begin{figure}[H]
\includegraphics[scale=0.8]{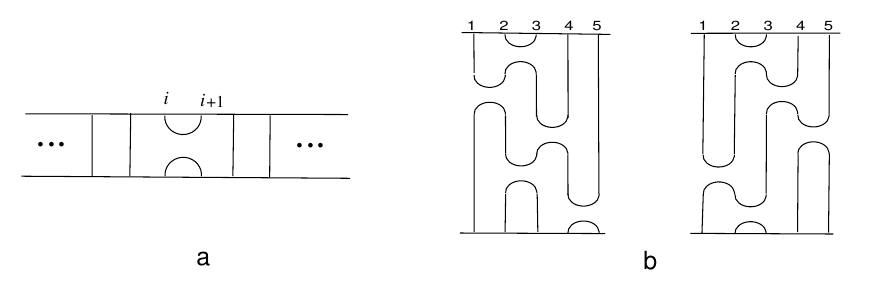}
\caption{a) The diagram of the tangle generator $\T_i$. b) The diagrams of  the normal  forms $ T^+(w)$ and $T^-(w)$  for  the  element in Figure \ref{Fig01}a. }\label{Fig02}.
\end{figure}
\end{example}

\begin{remark}\label{short} The  number of  generators $t_i$ in the normal forms is  minimal  in the set of  the representatives of  the same  element in $J_n$.  Moreover,  each  vertical  segment in the diagram of  a normal  form,  lying on  the $i$-{\it th} vertical segment    of the  identity, terminates either on the top or bottom line, or  on a  tangle $t_i$  or $t_{i-1}$. If  this  segment  terminates on two  tangles,  these  tangles  are necessarily different, i.e., $t_i$ on one  end  and   $t_{i-1}$ on the other  end.
\end{remark}

 \begin{proposition}\label{endpoints}  If $g$ is  a  gap of $w\in J_n$, then  the  diagram  of $w$ has  a  line  with  endpoint  $g$ in both  top and bottom segment,  and  the diagram of the normal  form  has  a vertical  line at $g$. Moreover,  the integers  $\underline j_1,\dots,\underline j_k$    are  either the left points  of     down-brackets or  the lower points of  ne-lines; the integers  $\overline i_1,\dots ,\overline i_h$    are  either the left points  of    up-brackets or  the upper points of  nw-lines.
 \end{proposition}
 \begin{proof} The diagram of the normal  form coincides   with that of the identity on the gap points.  The  first statement follows from the fact that  the    arc endpoints  are   invariants  of the set of representatives of the same element of the  monoid.  Consider now  a monomial  $U_{ij}$  and  observe that  the generator $t_j$  in it is the  last $t_j$ in $T^{+}(w)$. Since  the  generators at  right of $t_j$ in $T^{+}(w)$ have indices  higher than $j$, and so the corresponding tangles  are  at  right of the tangle  $t_j$ in the diagram,  the  left lower point   of this tangle is straightly connected to  the point  $j$ of the bottom segment. So it is  either  the  left point of a down-bracket or the lower point of  a  ne-line. Similarly, the generator $t_i$  in $V_{ij}$  appears  in $T^{-}(w)$ for the first time.  Since  the  generators at  left of $t_i$ in $T^{-}(w)$ have indices  higher than $i$, and so the corresponding tangles  are  at right of the  tangle $t_i$, the  left upper point  of this tangle  is straightly connected to  the point  $i$ of the top segment. So it is  either  the  left point of an up-bracket or the upper point of  a  we-line.
 \end{proof}

\begin{example}\rm See  Figure \ref{Fig02}b.  By (\ref{normalform}),  we  have $\underline j_1, \dots ,\underline j_3=1,2,4 $  and  $\overline i_1, \overline i_2= 2,1$. In the  bottom line of the diagrams in (b), 1 is the lower points  of a ne-lines and 2,4 are  left points of down brackets, while in the top line  1,2 are left points of  up brackets.
\end{example}

We are now ready to introduce and study the framed Jones monoid.

\begin{definition}
 We define the framed Jones monoid, denoted by $J_{d,n}$, as the monoid generated by $t_1, \ldots ,t_{n-1}$, $z_1,\ldots , z_n$ satisfying the relations (\ref{J01})-(\ref{J02}) and  (\ref{PreCdn}) together the following mixed relations.
\begin{align}
t_iz_i  = t_iz_{i+1}, \quad &z_i t_i  = z_{i+1} t_i,\label{Z2}\\
z_i t_j =  t_jz_i \quad &\text{for $i\not=j,j+1$,}\label{Z3}\\
  t_iz_i^k t_i  =  &  t_i.\label{Z4}
\end{align}
\end{definition}
Our next goal is to   provide a diagrammatic version of  $J_{d,n}$, which we call  the {\it abacus Jones monoid} and denote by  $\mathring{J}_{d,n}$.

\begin{definition}
     The abacus  Jones monoid $\mathring{J}_{d,n}$ is formed by the  diagrams of $J_n$ whose arcs have at most $d-1$  beads  sliding  on each arc. 
The product in  $J_{d,n}$ is defined by the usual concatenation,  by counting  beads modulo $d$, and   by neglecting the  loops containing or  not  beads.
\end{definition}
\begin{example}\rm See  Figure  \ref{Fig03}
\begin{figure}[H]
\includegraphics[scale=0.8]{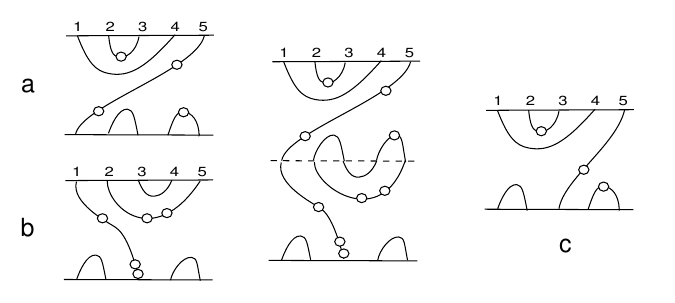}
\caption{  Two elements of $\mathring{J}_{4,5}$ and their  product. }\label{Fig03}
\end{figure}

\end{example}

\begin{remark} \label{Jones} Observe that the abacus monoid $C_{d,n}$ is a submonoid of $\mathring{J}_{d,n}$. Also,  the set formed by the tangle diagrams without beads,  $\mathring{J}_{1,n}$,    is a  submonoid of $\mathring{J}_{d,n}$
for every $d$. This submonoid  coincides with the Jones monoid $J_n$.
The  cardinality of   $\mathring{J}_{d,n}$
 is evidently  $d^n \cdot Cat_n$.
\end{remark}


\begin{lemma}\label{Lem01}
Let
$\mathring{w}\in \mathring{J}_{d,n}$, and   let $w$ be obtained from $\mathring{w}$ by removing the  beads. Let $g_1, \ldots , g_p$ be the gaps
of $w$ and $\overline  i_1,\dots, \overline i_h$, $\underline j_1,\dots, \underline j_k$ 
associated to $w$
 according to  Definition \ref{Not2}.
Then $\mathring{w}$ can be  written uniquely  in the following normal form
  \begin{equation} \label{dotnormalform} \OO_{g_1}^{q_1} \cdots  \OO_{g_p}^{q_p} \ \OO_{\overline i_1}^{r_1}\dots \OO_{\overline i_h} ^{r_h} \
T^{+}(w)
  \  \OO_{\underline j_1}^{s_1}  \dots \OO_{\underline j_k} ^{s_k}, \end{equation}
 where  $q_i,r_i,s_i\in\ZZ/d\ZZ$.
  \end{lemma}

\begin{proof} Given $\mathring{w}\in \mathring{J}_{d,n}$, let's move  the beads on each arc  to its endpoint at  left. For  a  ne-line this  endpoint is  on the  bottom line  and  for  a nw-line  it is on the  top  line.  In a  vertical  line  move  the beads to the  top.  So,  $\mathring{w}$ results in this  way   the product $\mathring{w}= o w o'$  where  $o$  and $o'$  are  products of  $\OO_i$,  and  $w\in J_n$.  Now, $w$ has  a  unique  normal  form  $T^{+}(w)$,   the  indices  $i$ of  the $\OO_i$  are  uniquely defined by  Proposition \ref{endpoints}, and the  exponents $q_i,r_i,s_i$ count modulo $d$ the  beads on each arc.
\end{proof}

\begin{example}\rm \label{exa5} The normal  form    for  the  element in Figure \ref{Fig03}a is $\OO_2 \T_2 \T_1 \T_3 \T_2 \T_4 \OO_1^2 \OO_4$, see its diagram in Figure \ref{Fig05}.

\begin{figure}[H]
\includegraphics[scale=0.8]{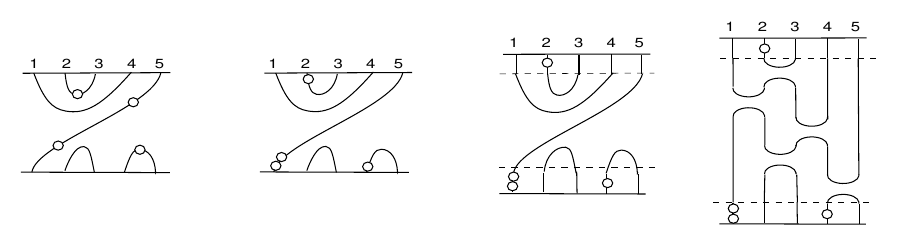}
\caption { The diagram  of  the normal  form   for  the  element in Figure \ref{Fig03}a.}\label{Fig05}
\end{figure}

\end{example}

\begin{remark} We have  $p+h+k=n$, since   the  numbers of  up-brackets,   down-bracket,  ne-lines, nw-lines  and  vertical lines  is  evidently $n$. Therefore, for  every  $w$,  there are $d^n$ possible choices of the  exponents $q_i,r_i,s_i$.   Thus, the  cardinality of
  $\mathring{J}_{d,n}$
 is  recovered.
\end{remark}

 \begin{theorem}\label{propJdn} The map \begin{equation}\psi: \label{assgn}t_i\mapsto \T_i,\,z_j\mapsto \OO_j, \quad i\in [ 1, n-1], j\in [1, n],\end{equation} extends to a monoid isomorphism $J_{d,n }\to \mathring J_{d,n}$. In particular, 
 $|J_{d,n}| = d^n Cat_n$.
  \end{theorem}
   \begin{proof} Consider  the monoid generated by 
  $$\{\T_1,\ldots,\T_{n-1},\OO_1,\ldots,\OO_n\}=\psi(\{t_1,\ldots,t_{n-1},z_1,\ldots,z_n	\})$$ with relations  \begin{align}
   \T_i^2  &=\T_i,   \label{O01} \\
 \T_i\T_j   &=\T_j\T_i \quad \text{if $\vert i-j\vert>1$, and }
\quad \T_i\T_j\T_i   =\T_i \quad \text{if $\vert i-j\vert=1$.} \label{O002}\\
 \OO_i^d  &=1,\quad \OO_i \OO_j = \OO_j \OO_i,\label{O1}\\
\T_i\OO_i  &= \T_i\OO_{i+1}, \quad \OO_i \T_i  = \OO_{i+1} \T_i,\label{O2}\\
\OO_i \T_j &=  \T_j\OO_i \quad \text{for $i\not=j,j+1$,}\label{O3}\\
\T_i\OO_i^k \T_i  &=    \T_i. \label{O4}
 \end{align}
 Note that \eqref{O01}-\eqref{O4}  are the images via $\psi$ of the  relations  \eqref{J01}-\eqref{J02}, \eqref{Z2}-\eqref{Z4} defining $J_{d,n}$.  In particular, $\psi$ extends to a monoid map $J_{d,n }\to \mathring J_{d,n}$.  The map $\psi$ is  surjective because  the elements  $\T_i, \OO_i$
generate $ \mathring J_{d,n}$ by Lemma \ref{Lem01},  and  by observing  that    different normal  forms  are   images by $\psi$ of different  elements  of  $J_{d,n}$.   The latter claim  follows from the  fact that  in absence  of   generators  $\OO_i$  the  normal  forms    of  $J_n, \mathring J_{d,n}$ coincide and that by the  relations \eqref{O01}-\eqref{O4} a  bead  neither can be created nor can jump  from  a  strand  to another one. Finally, in the next Lemma \ref{Lem02} we prove  that every  word in the  generators  $\T_i$  and $\OO_i$ can be put, by using all  relations  (\ref{O01})--(\ref{O4}), in the  normal form defined in Lemma  \ref{Lem01}.
 This shows that $|J_{d,n }|\leq |\mathring J_{d,n}|$ and finishes the proof.
\end{proof}

 \begin{lemma}\label{Lem02} Every word  in the  generators  $\T_i$ and $\OO_i$ can be put, in $\mathring{J}_{d,n}$, in the  normal form  (\ref{dotnormalform}).
  \end{lemma}

\begin{proof} Let $ w$ be a word in the generators $\T_i$ and  $\OO_i$. Observe  that  by  definition of  concatenation,  the diagram of $ w$  is  an  element of $\mathring J_{d,n}$  and  hence it has  a normal  form.  We call t-length of  $  w$ the  number of  generators $\T_i$ in it. Now, if  $ w$ does not contain generators  $\OO_i$, its t-length can be  decreased till  the minimal  length of the normal form, by using the  relations (\ref{J01})-(\ref{J02}), by  Remarks \ref{short} and \ref{Jones}.  We prove that this can be done  also in presence of  generators  $\OO_j$, by using the relations of $\mathring{J}_{d,n}$.  To  reduce the t-length, we have to use  the second relation in (\ref{J02}) and relation (\ref{O2}); i.e, we have to collect  together    generators forming the  subwords
\begin{equation}\label{subwords}\T_i \T_{i\pm 1}\T_i,  \quad \T_i \OO_i^k \T_i , \quad   \T_i \OO_{i+1}^k \T_i, \quad  k\in  \ZZ/d\ZZ,\end{equation}
which  will be  replaced by $\T_i$.   To do  that, suppose we have to move $\T_i$ at right of  a subword $ v$, and that  this  is allowed in absence of the $\OO_j$ in $v$. This  means $ v$   contains only generators $\T_k$ with  $|k-i|>1$. If each $\OO_j$ in $ v$ satisfies  $j\not=i, i+1$, then  we  get  $\T_i   v=  v \T_i $  by  (\ref{O3}). So, $\T_i$ cannot move at  right in $  v$ if it meets either $\OO_i$ or $\OO_{i+1}$.  In both these  cases, by (\ref{O2}), we consider $\T_i \OO_i$  as a binomial  that has to be moved  at  right  and  so on. Because of relations (\ref{O1}), we get  eventually  $\T_i   v =  v' \T_i \OO_i^k$, $k\in \mathbb{Z}_d$. If $  v$ was  followed  by $\T_i$, we are done. If $ v$ was  followed  by $\T_{i+1} $, and $k>0$, then  we write   $\T_i \OO_i^k \T_{i+1}  =  \T_i \OO_{i+1}^k  \T_{i+1}  =\T_i \OO_{i+2}^k  \T_{i+1}   =  \OO_{i+2}^k \T_i \T_{i+1}  $. Similarly, if $  v$ was  followed  by $\T_{i-1}  $,    we write  $\T_i \OO_i^k \T_{i-1}  =  \T_i \OO_{i+1}^k  \T_{i-1}  =\T_i   \T_{i-1} \OO_{i+1}^k $.  We will follow an  analogous procedure to  move  $\T_i$ at left of  a  word  $v$. Therefore,   we can collect  together the generators  to form the  subwords in (\ref{subwords}).  Proceeding in this  way,  we get a  word  $ w'$  which  contains  exactly the  same $\T_i$  generators  as  the  normal form of  $w$.

Now we have to collect  the $\OO_j$ at  the beginning and at the  end of the  word.  Observe  that, if  there  is some  $\OO_j$ such that  $j$ is  a  gap, such $\OO_j$ commutes  with all other  generators in $w'$  and  then it can be  moved  at  the beginning of the word. If  $j$ is not a  gap, the moving of  $\OO_j$  at  right or  at  left in the  word, is stopped  when it meets  $\T_j$ or $\T_{j-1}$.  Observe that, diagrammatically, the move of $\OO_j$  corresponds  to  a  sliding of  the  bead up and  down  along a  vertical  segment.  Therefore, we start  from the first $\OO_j$ in $w'$ which is  not a  gap  and proceed  as follows. Our  aim is  to decrease the  index of $\OO_j$  by relation (\ref{O2}). We consider  one  case, the others  are obtained  by  exchanging left with right.   If  moving  $\OO_j$ at  right we  meet   $\T_j$,  we move it at  left.  Then  either it  reaches  the  beginning of the  word,  or,  by  Remark  \ref{short}, it meets $\T_{j-1}$.  Then $\T_{j-1}\OO_j$ is  replaced  by $\T_{j-1} \OO_{j-1}$.  Again we  move  $\OO_{j-1}$ at right. It can now overpass $\T_j$.  Then, either  it  reaches  the  end of  the word, or it meets $\T_{j-2}$, by the same reason. So, $\OO_{j-1}\T_{j-2}$ is  replaced by  $\OO_{j-2}\T_{j-2}$ and  we  move  $\OO_{j-2}$  at left  and so on.  Eventually  we  get  $\OO_m$, with $m<j$,  at  the  beginning or  at  the  end  of  the  word.   This is done  for  each  $\OO_j$ in $ w'$, till  $  w'$ is  transformed  in a  product  $o w'' o'$,  such  that $w''$ does not contain $\OO_j$ and  can be  put in the  normal  form $T^+(w'')$. Moreover,  since  our procedure  has  only diminished  the indices  of  the $\OO_j$,   $o$     contains  $\OO_i$  only if   $i$ is  a gap, or  is  the left endpoint of  an up-bracket or  the  top  endpoint of  a nw-line,  and $o'$     contains  $\OO_i$  only if  $i$ is  the left endpoint of  a down-bracket or  the  lower  endpoint of  a ne-line.  The  rearrangements of  the  $\OO_i$  in $o$  and  $o'$ is made  by  using relations (\ref{O1}).  The normal  form  (\ref{dotnormalform}) is  attained.
\end{proof}

 \subsection{The Framed Brauer  Monoid }

 \begin{definition}
The framed Brauer monoid  $Br_{d,n}$ is the monoid   presented by generators  $s_1,\ldots ,s_{n-1}$, $t_1,$ $ \ldots, t_{n-1}$, $z_1, \ldots , z_n$ satisfying (\ref{B01})-(\ref{B03}), the defining relations of $S_{d,n}$
(Remark \ref{RemPre}) and $J_{d,n}$.
 \end{definition}

As in the case of the Jones monoid, we will introduce below,  in Definition \ref{DefAbacusBrauer}, 
the  abacus version $\mathring Br_{d,n}$ of $Br_{d,n}$. To do this, we first observe that the elements of 
  $Br_n$  can be  represented  by  a  diagram  with  $n$  arcs  with  endpoints  on  two  horizontal  lines, the  only  difference  being  that  here  the  lines  can  cross  each other.
  It is known that this monoid admits the  presentation displayed in Subsection \ref{subsectionJB}.\par
  We  recall  that
 in  \cite[Lemma 28]{AiArJuJPAA2023},  the  monoid $ Br_n$  has been provided  with a normal form for its elements, namely
\begin{equation}  s  \T_1\T_3 \dots \T_{2k-1}  s',  \label{nformbrauer}\end{equation}
where   $0\le k\le n/2$  is  the  number of up-brackets, and   $s,s'\in  S_n$.

The  normal  form is defined  in the following  way.
Let  $1,\dots, n$  be the points of  the top segment.  We  define the  set  $A\subset \ldbrack 1,n\rdbrack$ by putting $j\in A$ if  $j$ is the left point  of  an up-bracket. The  cardinality of  $A$ is $k\le n/2$. The  set  $A$ is  ordered by  the order in $\ldbrack 1,n\rdbrack$.  Then  we define the ordered  set  $B\in \ldbrack 1,n\rdbrack$  containing  the  respective right  endpoints of  the up-brackets.  In the  third set  $C\in \ldbrack 1,n\rdbrack$ we put  the $n-2k$ upper points of  the  lines, ordered  by their order  in $\ldbrack 1,n\rdbrack$.  For the   points $1,\dots, n$  in the bottom segment, we  define  similarly: the  ordered set  $A' \subset \ldbrack 1,n\rdbrack$ by putting   $j\in A'$ if  $j$ is the left point  of  a down-bracket, keeping the order of  $j\in \ldbrack 1,n\rdbrack$; the   ordered   set $B'$ containing  the    $k$  right  endpoints respecting the order in $A'$;  and  the  ordered  set  $C'$ containing the  lower points of  the  lines,    respecting the order  of the corresponding upper points in $C$.
Then  we  write the  normal  form as  (\ref{nformbrauer}),
where  $s$  is  the permutation  sending  the  $a_i\in A$ to  $2i-1$ and  $b_i\in B$ to $2i$, for  $i \in \ldbrack 1,k \rdbrack$, and  $c_i\in C$ to  $2k+i$  for  $i\in \ldbrack 1, n-2k\rdbrack$.  The  permutation  $s'$  sends $2i-1$ to $a'_i\in A'$ and  $2i$ to $b'_i\in B'$, for  $i=1,\dots, k$, and   $2k+i,\dots, n$ to $c'_i\in C'$, see  Example \ref{nf} and the following figure.
 \begin{figure}[H]
 \includegraphics[scale=0.7]{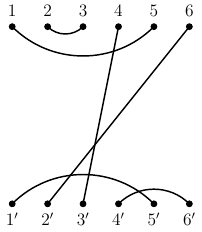}
 \quad
  \includegraphics[scale=0.7]{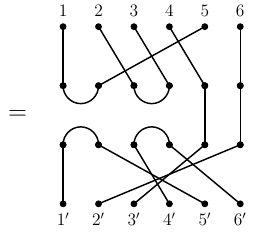}
 \caption{An element of $Br_6$  and its normal  form.}
 \end{figure}
 \begin{definition}\label{DefAbacusBrauer}
 We denote by  $\mathring{Br}_{d,n}$ the {\it abacus Brauer  monoid}, which is  formed by the  diagrams with $n$   arcs, with at most $d-1$  beads  sliding  on each arc.   Also  here the  product   is defined as usually  by concatenation, and  again by counting  beads modulo $d$  and neglecting the  loops containing or  not  beads, and containing or not crossings.    
\end{definition}

 Note that $\mathring{Br}_{1,n} =  Br_n$  is a submonoid of $\mathring{Br}_{d,n}$ for all $d,n$.
\begin{example}\rm See  Figure  \ref{Fig08}
\begin{figure}[H]
\includegraphics[scale=0.8]{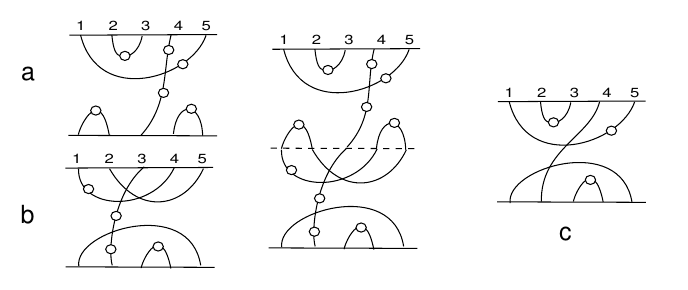}
\caption{  Two elements of $ \mathring{B}_{4,5}$ and their  product. }\label{Fig08}
\end{figure}
\end{example}

\begin{theorem}\label{Brauerabacus} The monoid   $\mathring{Br}_{d,n}$
 is generated by   $\T_1, \dots, \T_{n-1}$,  $\s_1,\dots, \s_{n-1}$, $\OO_1,\dots, \OO_n$,  subject to the  relations (\ref{Coxeter-Moore1}), (\ref{Coxeter-Moore2}), (\ref{J01})-(\ref{J02}), (\ref{B01})-(\ref{B03}),
(\ref{O1})-(\ref{O4}), and
 \begin{equation}\OO_j \s _i =   \s_i \OO_{\s_i(j)}.\label{OS}\end{equation}
  In particular, $\mathring{Br}_{d,n}$ is isomorphic to ${Br}_{d,n}$, which has therefore cardinality $(2n-1)!!d^n$.
 \end{theorem}
\begin{remark} A more formal statement, along the lines of Theorem \ref{propJdn} , would be: the map 
\begin{equation}\label{phi}\phi: s_i\mapsto  \s_i, t_i\mapsto \T_i, z_i\mapsto \OO_i\end{equation} extends to a monoid isomorphism ${Br}_{d,n}\to \mathring{Br}_{d,n}$.  \end{remark}

Observe that, diagrammatically, an element of   $S_n$  is  represented by  a diagram  with $n$  arcs, all having the upper points  on the  top segment  and  the  lower point on the bottom segment. Except that for  the  identity element, all  elements  have  diagrams  with   crossing  arcs.
 Relations (\ref{OS}) between  framing generators  and  elements  $\s_i$,   (see Remark \ref{RemPre}) mean, in term of  diagrams, that  the beads  generators   $\OO_i$    are sliding along  the  arcs  also in presence of crossings.   This  does not  contradict  a  topological  interpretation of  the diagrams as plane projection of  physical strings as in the  case of  the  braids.  Indeed,  we  may interpret the element $\s_i$  as  a  positive  crossing,  but  relation \eqref{Coxeter-Moore1} means that    positive and  negative  crossings  are  indistinguishable.

 Since the diagrams $\T_i,\OO_i,\s_i$ clearly satisfy the relations in the statement of Theorem \ref{Brauerabacus}, to complete its proof it suffices to  show   that the elements of  $\mathring{Br}_{d,n}$  admit  a normal  form  in terms of  the  generators,  and then that  every  element   can be  put in that  normal  form.
This is  achieved in the next three lemmas.

\begin{lemma}\label{Lem03}  Let
$\mathring{w}\in \mathring{Br}_{d,n}$, and   let $w$ be obtained from $\mathring{w}$ by removing the  beads.
Then $\mathring{w}$ can be  written uniquely  in the following normal form
  \begin{equation}
   o  N(w)  o',\label{nformbrauerabacus}
  \end{equation}
 where  $N(w)$ is  the normal  form of $w$ given by (\ref{nformbrauer}), $o$ is the product of  $\OO_j^{\epsilon_j}$ over all  $j$ on the upper  segment that are  either  left points of  up-brackets or upper points of  lines,  $\epsilon_j$  are the numbers of  beads on the  corresponding  arcs. Similarly, $o'$  is the product of  $\OO_j^{\epsilon'_j}$  over  all  $j$ that  are  the  left points of  down-brackets, with  $\epsilon'_j$  beads  respectively.
\end{lemma}
\begin{proof} Given an element   $\mathring{w}\in \mathring{Br}_{d,n}$, it is  evident that, in each bracket, all  beads  can be moved near the  left  endpoint,  and in each line, all  beads can be moved   near  the upper point. Then, the diagram  results the concatenation of   three  diagrams:   the  first one at top  and  the last one at  bottom containing only beads, and a  central  diagram containing arcs  without  beads.  This last  central diagram is put  in the  normal  form $N(w)$.
\end{proof}

\begin{remark}\rm Observe that every  $x\in J_n$  has no crossings.  However, it lives also in $Br_n$, and   its normal  form in $Br_n$  may have crossings. The  same remains true for  $x\in \mathring{J}_{d,n}$, see  Figure \ref{Fig09}.
\end{remark}

\begin{example}\label{nf} To  write  the  normal  form  for  the  element  in Figure  \ref{Fig09},  we  define $A=\{1,2\}$,  $B=\{3,4\}$, $C=\{5\}$,    $A'=\{2,4\}$,  $B':=\{3,5\}$, $C'=\{1\}$. Moreover, $\epsilon_1=0,\epsilon_2=1, \epsilon_5=2, \epsilon'_4=1$.  The  normal  form is  $$\OO_2 \OO_5^2 \  s \ \T_1 \T_3 \ s' \  \OO_4 ,$$
where  $s=(^{1,2,3,4,5}_{1,4,2,3,5})=\s_3 \s_2$  and  $s'=(^{1,2,3,4,5}_{2,3,4,5,1})=\s_4 \s_3 \s_2 \s_1$, see Figure \ref{Fig09}.
\begin{figure}[H]
\includegraphics[scale=0.8]{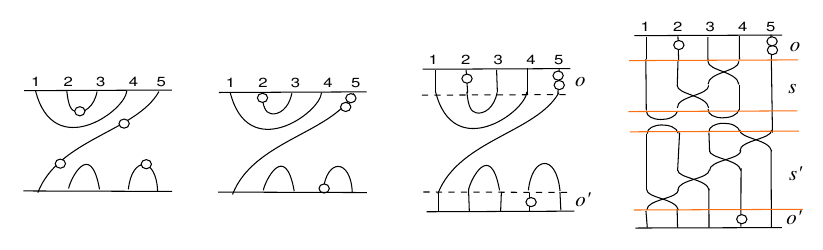}
\caption{The normal  form in $\mathring{Br}_{4,5}$ of the element  in Figure  \ref{Fig05}.}\label{Fig09}
\end{figure}

\end{example}

\begin{note}\rm  In what follows  we  omit for short  the  exponents  $\epsilon_j$ of the $\OO_j$ since  they do  not affect the  proofs.
\end{note}

\begin{lemma}\label{Lem04}
Every  word  in the  generators  $\T_i$, $\s_i$ and  $\OO_i$  can be  put  in the  normal  form  (\ref{nformbrauerabacus}) by using the relations given in Proposition \ref{Brauerabacus}.
\end{lemma}
\begin{proof}   Now,  suppose that  $\mathring{z}$ is a  word  in the  generators  $\s_i, \T_i$ and $\OO_i$, and  let $z$ be the  word obtained  from $\mathring{z}$ by canceling all  $\OO_i$.  We know  (\cite[Theorem 3.1]{KuMaCEJM2006})  that  every  word in the generators $\T_i$  and  $\s_i$ can be  put in the normal  form  (\ref{nformbrauer}) by  using the  relations  defining $Br_n$, so  $z$  can be put in the normal form, that we  denote   $N(z)$.   Observe that  the  transformation of  $z$ into $N(z)$ requires a  chain of, say, $m$ of local  transformations, each one  using  a  relation among  (\ref{J01})-(\ref{J02}) and (\ref{B01})--(\ref{B03}), here  denoted  $R_k$.  Observe that any  relation $R_k$  is of the form  $v=v'$,  where $v$ and  $v'$ are  two  words defining the  same   element;       we  will  write indifferently  $v=\rho_k(v')$ or $v'=\rho_k(v)$.  Let  $z_0=z$,  $z_i$ the words at intermediate steps and     $z_m=N(z)$. At each  step,  $z_i$ can be  written as  product of  subwords $x_i v_i  y_i$,  where  $x_i$ and  $v_i$ may  be  1,  and  $z_{i+1}=  x_i  \rho_k(v_i) y_i$.   Now  we  come  back to  $\mathring{z}$, we  write   $\mathring{z}_0= \mathring{z}$   and we  subdivide  it  into  subwords  $\mathring{x}_0  \mathring{v}_0 \mathring{y}_0$, in such a way  that $\mathring{v}_0$ has  no  generators  $\OO_j$  at  the  beginning  and  at  the  end, i.e.,  these  generators  are put in $\mathring{x}_0$  and  $\mathring{y}_0$ respectively.   So,  the  generators  $\OO_j$ may  be inside the word  $\mathring{v}$.  In the  case  they  do not occur, i.e., $\mathring{v}_0=v_0$,  we can  proceed with   $\mathring{z}_1= \mathring{x}_0 \rho_k(\mathring{v}_0) \mathring{y}_0$  and  so on. The problem  arises  when  $\mathring{v}_i$  has  generators  $\OO_j$ inside.
We  need  to prove the  following result.
\begin{lemma}\label{relationsBr} Let  $w=w'$  be   the image via $\phi$ (cf. \eqref{phi})  of a  relation  $R_k$ of $Br_n$.  Let  $\mathring{w}$  be obtained   by inserting  products of generators  $\OO_j$,  all non commuting  with  $w$,  between  every pair of  adjacent   generators  in the  word  $w$.  Then,  in $\mathring{Br}_n$,   we have $\mathring{w}= o_1 w o_2$, with  $o_1$ and $o_2$ products of $\OO_j$, so that  relation $R_k$  can be applied  to $\mathring{w}$,  giving  $\mathring{w} = o_1 w' o_2$.
\end{lemma}
\begin{proof}  The proof  consists in verifying all  relations. We  give  here two examples:

1) For   relation (\ref{B03}), with $i=1, j=2$,  we  write
$$\T_1 \OO_1 \OO_2 \s_2 \OO_1 \OO_2 \T_1 \stackrel{(\ref{O2})}{=}  \T_1 \OO_2^2  \s_2 \OO_2^2 \T_1 \stackrel{(\ref{OS})}{=} \T_1 \OO_3^2 \s_2 \OO_3^2 \T_1 \stackrel{(\ref{O3})}{=}  \OO_3^2 \T_1 \s_2 \T_1 \OO_3^2\stackrel{(\ref{B04})}{=} \OO_3^2 \T_1 \OO_3^2. $$

2) For relation (\ref{B04}), with $i=1, j=2$,  we    write
$$  \T_1 \OO_1 \OO_2 \OO_3 \T_2  \stackrel{(\ref{O2})}{=}  \T_1  \OO_1^2 \OO_3 \T_2\stackrel{(\ref{O3})}{=}   \OO_3 \T_1 \T_2 \OO_1^2 \stackrel{( \ref{B03})}{=}  \OO_3 \T_1 \s_2 \s_1\OO_1^2 \stackrel{(\ref{B04})}{=} \OO_3 \s_2\s_1\T_2 \OO_1^2. $$
\end{proof}
By using  the  lemma  above  for the  subword  $\mathring{v}_i$    if it contains  generators  $\OO_j$, we   finally get  a  word   $\mathring{N}(z)$ which coincides  with $N(z)$ if all  generators  $\OO_j$  are removed. If  $N(z)$ has no generators  $\T_i$, then  the  $\OO_j$  can be  moved  at left of  the  word by using (\ref{OS}) and  the  normal form (\ref{nformbrauerabacus}) is  attained.  Observe that  we  use  the  word  {\it move}  having in mind the  bead that slide  along the  arc,  but in  fact by  overpassing a   crossing $\s_i$, $\OO_j$ is transformed into $\OO_{\s_i(j)}$ by (\ref{OS}).   If  $N(z)$  contains  $\T_i$,  then  write  $N(z)=s T s'$ where  $T=\T_1 \T_3\dots \T_{2k-1}$ (see  \eqref{nformbrauer}). Now,  we want  to
 move the   $\OO_j$   in such a  way  that  $\mathring{N}(z)$ is  transformed into  a  word  of  type   $ s O T O' s'$.   More precisely,  we  do the following:  in  $\mathring{N}(z)$  we  denote $\mathring{T}$  the subword that  coincides with $T$  by  removing  the generators  $\OO_j$.  Now, let  $\OO_j$ appear on the  left of  $\mathring{T}$, but also on the left of  generators  $\s_i$  of $s$.  We  can move it  at right till it becomes $\OO_r$ close to  $\mathring{T}$. If  $r=2i\le 2k$, then we  replace $\OO_r$ by $\OO_{r-1}$ using (\ref{O2}).  We repeat  this procedure for  every  $\OO_j$ at left of $\mathring{T}$.  Similarly, let  $\OO_j$ on the  right of  $\mathring{T}$, but also on the   right  of  generators  $\s_i$  of $s'$.  We  can move it  at left till it becomes $\OO_q$ close to  $\mathring{T}$. In this  case, if $q=2i\le 2k$, then we  replace $\OO_q$ by $\OO_{q-1}$ using (\ref{O2}),  while    if  $q>2k$,  we  put $\OO_q$  at left of  $\mathring{T}$  by (\ref{O3}). We repeat  this procedure  for  every  $\OO_j$ at right of $T $. It remains to consider  the case    when $\OO_j$ is  inside $\mathring{T}$. If $j>2k$, $\OO_j$ is  moved  at left of $\mathring{T}$  by using (\ref{O3}).     If  $j= 2i-1$, or $j=2i$,  and $\OO_j$ is at right (resp. at left)  of $\T_{2i-1}$, then  we put $\OO_{2i-1}$ at  right (resp. at left) of  $\mathring{T}$,  by (\ref{O2})  and (\ref{O3}).
  In this  way we   obtain  a  word of type  $s  O   T  O'  s'$. Observe that if $O$ contains    $\OO_j$, then either $j>2k $   or  $j$ is odd, while if $O'$  contains   $\OO_j$ then  $j<2k$  and  odd.
  Therefore,     by  applying $s^{-1}$ to all indices  $j$  in  $O$  and  $s'$ to all  indices $j$ in $O'$,  we get   respectively $o$ and $o'$,  so the  normal  form  (\ref{nformbrauerabacus})   is  obtained.
\end{proof}

\subsection{Two Framed Rook Monoids}

Recall from Subsection \ref{RM} the rook monoid $R_{n}$ and  its presentations.   Also recall that the rook monoid admits a geometric  realization \cite[Subsection 2.2]{AiArJuMMJ}: an  element  of  the  monoid  $R_n$  can be  represented  by  a  diagram  with  $k\le n$  lines  all having the upper points  on the  top segment  and  the  lower point on the bottom segment.   The $2n-k$ points  of the  top  and of the bottom segments that  are not  endpoints of  lines  are  endpoints of  arcs having the other  endpoint  moving  freely  between the  top  and  bottom  segments. These  arcs,  called  {\it broken arcs}, may of course shrink to their non-free endpoint  on the top or  bottom  segment.

Before  introducing  a presentation  for  the two framed Rook monoids,  let's recall  the  diagrams  of  the  generator   $r_i$  and  of  the  generator   $p_i$, see  Figure  \ref{Fig13}.  Observe that  relation (\ref{RHO2}) is  evident.

\begin{figure}[H]
\includegraphics {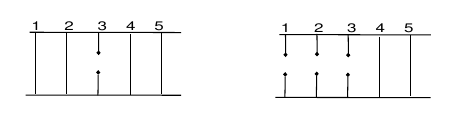}
\caption{  The  generators  $r_3$  and  $p_3$  of  $R_5.$ }\label{Fig13}
\end{figure}

  In  \cite[Remark 3]{AiArJuMMJ},  the  monoid $R_n$    has been provided  with a normal form, namely
\begin{equation} r_{i_1} r_{i_2} \cdots r_{i_{n-k}}  s,  \label{nformrook}\end{equation}
where   $0\le k\le n $  is  the  number of lines, $i_1,\dots, i_{n-k}$  are  the points  on the top  segments  that  are  not endpoints of lines  and   $s\in  S_n$  sends the upper points of  the  lines to their lower points  and  the remaining  $n-k$ ordered points  on the  top  segment  to  the $n-k$ ordered points of the bottom  segment, e.g.  see  Figure  \ref{Fig10}.
\begin{figure}[H]
\includegraphics[scale=0.9]{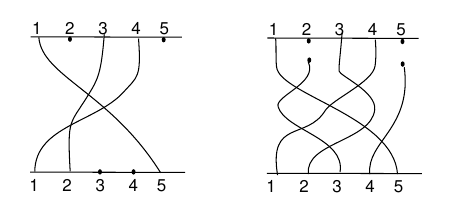}
\caption{An element of $R_5$  and its normal  form.}\label{Fig10}
\end{figure}

\begin{remark} Observe that for  the  rook monoid the elements   are
either point  or lines. Now, it is  natural to  attach  a  framing to
a  line  as in the  symmetric  group, while it is less natural  to
attach  a framing to  a  point.  In literature indeed (e.g.
\cite{ChEaJA2023, AiArJuMMJ})  the framed rook  monoid $R_{d,n}$ has
no framings  attached to  points.
Furthermore, we consider here the case where the points can be
provided with framing, thus defining the framed rook monoid
$R'_{d,n}$. Observe that  $R_{d,n} \subset \RR_{d,n}$.
\end{remark}
\subsubsection{The    framed rook monoid  $ R_{d,n}$ }\label{Rdn}
\mbox{}
\begin{definition}
The   framed rook monoid $R_{d,n}$  is defined by generators $s_1, \ldots ,s_{n-1}$, $p_1, \ldots , p_1$, $z_1, \ldots , z_n$ satisfying the defining relations of $R_n$,  relations (\ref{Mrook1})-(\ref{Mrook4}), (\ref{sizi}), and the following relations:
\begin{align}
p_iz_j & = z_jp_i, \qquad\text{for $1\leq i< j\leq n$,}\label{fRdn1}\\
p_i z_j & = z_j p_i = p_i,\qquad\text{for $1\leq j\leq i\leq n$.}\label{fRdn2}
\end{align}
\end{definition}
\begin{remark}\label{remarkEast}
The above definition coincides with  the one given    in \cite[Theorem 3.22] {ChEaJA2023}, given in terms of generators $r_i$ with  relations \eqref{M2rook1}-\eqref{M2rook4}, so that  relations (\ref{fRdn1})  and  (\ref{fRdn2}) are  replaced by
\begin{align}
r_iz_j & = z_jr_i, \qquad\text{for  all $ i \not=j $,}\label{fRdn3}\\
r_i z_i &  = z_i r_i= r_i.\label{fRdn4}
\end{align}

\end{remark}

We shall prove that the cardinality of  $R_{d,n}$ is  given by  $\sum_{k=0}^n  { \binom{n} {k}}^2 k!  d^{k}$
(so, for  $n=1,\dots,7, d=2$  the  formula  gives  3, 17, 139, 1473, 19091, 291793, 5129307). 
We  proceed   as for the previous  framed  monoids.

 We denote by  $\mathring{R}_{d,n}$ the {\it first abacus Rook monoid}, which is  formed by the  diagrams with $k$  lines with at most $d-1$  beads  sliding  on each  line,  and  $2n-k$    broken arcs without  beads. Indeed, in  this first  abacus  Rook  monoid  the  free  endpoint  of  a broken  arc   allows the  beads to  escape. Also  here the  product   is defined as usually  by concatenation, and  again by counting  beads modulo $d$  and neglecting  arcs with  two  free endpoints  containing or  not  beads. Note that $\mathring{R}_{1,n} =  R_n$ and $R_n$ is a submonoid of $\mathring{R}_{d,n}$ for all $d,n$.

\begin{example}\rm See  Figure  \ref{Fig12}.
\begin{figure}[H]
\includegraphics[scale=0.8]{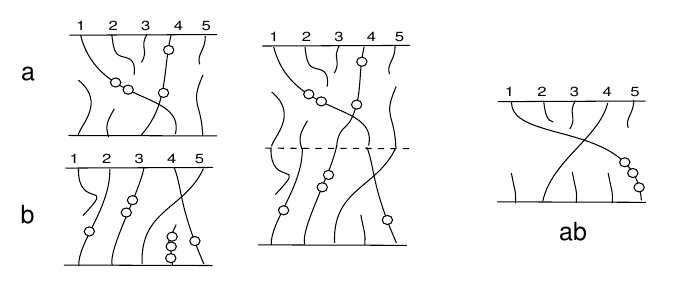}
\caption{  Two elements of $ \mathring{R}_{4,5}$ and their  product. }\label{Fig12}
\end{figure}
\end{example}

The  cardinality of  $\mathring{R}_{d,n} $  is evidently  $\sum_{k=0}^n  { \binom{n} {k}}^2 k!  d^{k}$.
Then,  we  prove the following

 \begin{theorem}\label{Rook1abacus} The monoid   $\mathring{R}_{d,n}$
 is generated by     $\s_1,\dots, \s_{n-1}$,  $r_1,\dots, r_{n}$, $\OO_1,\dots, \OO_n$,  subject to the  relations  (\ref{Coxeter-Moore1}),   (\ref{Coxeter-Moore2}),
(\ref{O1})-(\ref{O4}), (\ref{OS}),   (\ref{Mrook1})-(\ref{Mrook4})  and
 \begin{align}  & r_i \OO_j= \OO_j r_i  \qquad \text{for all $i\not= j$,} \label{OP1} \\
                 & r_i \OO_i= \OO_i r_i=   r_i. \label{OP2}
\end{align}Hence  ${R}_{d,n}$ is isomorphic to  $\mathring{R}_{d,n}$ and its cardinality is $\sum_{k=0}^n  { \binom{n} {k}}^2 k!  d^{k}$. 
 \end{theorem}
 As in the previous cases, we  prove that  every  element of $\mathring{R}_{d,n}$  can be  written  in terms of  the generators
$s_i$ , $r_i$ and  $\OO_i$  in the following  normal  form.
   \begin{equation} \OO_{j_1}^{m_1} \OO_{j_2}^{m_2} \cdots \OO_{j_k}^{m_k} r_{i_1} r_{i_2} \cdots r_{i_{n-k}}  s,  \label{nformfrook}      \end{equation}
where $j_1,\dots ,j_k$  are  the  upper points of  the $k$ lines,  $i_1,\dots, i_{n-k}$  and  $s$  are  defined as in (\ref{nformrook}).

Observe that   $m_1, \dots, m_k$  lie in $[0, d-1]$,  and  $s$ is a  permutation of $S_k$, and hence   $\mathring{R}_{d,n}$ has cardinality $\sum_{k=0}^n  { \binom{n} {k}}^2 k!  d^{k}$.
To conclude, we  verify that   that  every  word  in  such  generators, using the  relations  above,  can be put in  the normal form (\ref{nformfrook}).    We  get  that  $\mathring{R}_{d,n}$ is isomorphic to  $R_{d,n}$, so  Theorem \ref{Rook1abacus}  follows.

\subsubsection{The framed Rook monoid  $  \RR_{d,n}$ }

 \begin{definition}
 The  monoid  $\RR_{d,n}$  is defined by generators $s_1, \ldots , s_{n-1}$, $p_1, \ldots , p_n$, \break $z_1, \ldots , z_n$ satisfying the  same defining relations of $R_n$ except for   relation (\ref{fRdn2}) which is replaced by
\begin{equation}
 p_i z_1^{m_1} z_2^{m_2}\cdots z_i^{m_i}  p_j = p_j z_1^{m_1} z_2^{m_2}\cdots z_i^{m_i} p_i =  p_j, \qquad  1\leq i\le j\leq  n .   \label{Rdn2}
\end{equation}
\end{definition}

\begin{remark}
In  terms of  generators  $r_i$,  eq. (\ref{Rdn2}) becomes
\begin{equation}
r_iz_i^k r_i  = r_i. \label{Rdn3}\\
 \end{equation}
 \end{remark}

We  now introduce the  abacus  monoid   $\mathring{\RR}_{d,n}$, consisting of  diagrams  with $k$  lines with at most $d-1$  beads  sliding  on each  line,  and  $2n-k$    broken arcs with at most $d-1$   beads  sliding on each arc. The  free  endpoint of  a  broken arc  here blocks the  beads, and in the diagram it is marked by a  bar. The  product   is defined as usually  by concatenation, and  again by counting  beads modulo $d$  and neglecting  arcs with  two  free endpoints  containing or  not  beads. Note that  also in this  case  $\mathring{\RR}_{1,n} =  R_n$ and $ R_n$ is a submonoid of $\mathring{R}_{d,n}$ for all $d,n$.
\begin{example}\rm See  Figure  \ref{Fig11}.
\begin{figure}[H]
\includegraphics[scale=0.8]{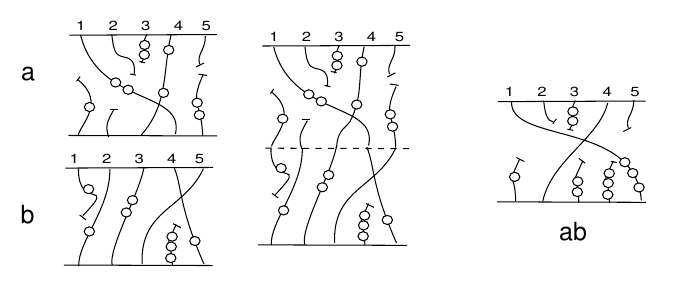}
\caption{  Two elements of $ \mathring{\RR}_{4,5}$ and their  product. }\label{Fig11}
\end{figure}

\end{example}

Since  for  every  $k\in  [0,n]$, there are at most $d-1$ beads on each  of the $k$  lines, as well as  on the $2(n-k)$ broken arcs,  the cardinality  of  $\mathring{\RR}_{d,n}$  is    $\sum_{k=0}^n { \binom{n} {k}}^2 k!  d^{2n-k}$.

 \begin{theorem}\label{Rookabacus} The monoid   $\mathring{\RR}_{d,n}$
 is generated by     $\s_1,\dots, \s_{n-1}$,  $r_1,\dots, r_{n}$, $\OO_1,\dots, \OO_n$,  subject to the  same relations  as  $\mathring{R}_{d,n}$   excluded   (\ref{OP2})  that  is replaced  by
 \begin{equation}
  r_i \OO_i^k  r_i=  r_i \qquad \text{for all  $i, k$.} \label{Or3}
 \end{equation}
 Hence  ${R'}_{d,n}$ is isomorphic to  $\mathring{\RR}_{d,n}$ and its cardinality is $\sum_{k=0}^n { \binom{n} {k}}^2 k!  d^{2n-k}$.
 \end{theorem}
 \begin{proof}
  We  prove that  every  element of $\mathring{\RR}_{d,n}$  can be  written  in terms of  the generators
$s_i$, $r_i$ and  $\OO_i$  in the following  normal  form.
   \begin{equation} \OO_{ 1}^{m_1} \OO_{ 2}^{m_2} \cdots \OO_{n}^{m_n} \ r_{i_1} r_{i_2} \cdots r_{i_{n-k}} \  s
  \  \OO_{j_1}^{q_1} \OO_{j_2}^{q2}\cdots \OO_{j_{n-k}}^{q_{n-k}} ,  \label{nformrookdn}      \end{equation}
where $j_1,\dots j_{n-k}$  are  the  endpoints of  the $n-k$ broken arcs on the bottom segment,  $i_1,\dots, i_{n-k}$  and  $s$  are  defined as in (\ref{nformrook}).

Observe that  the  exponents $m_1, \dots ,m_n$  and   $q_1,\dots, q_{n-k}$  lie  in  $\ldbrack 0, d-1\rdbrack$,  and  $s$ is a  permutation of $S_k$, and hence  the set of  normal forms  has  the  needed cardinality.
Proposition \ref{Rookabacus} thus results  by proving   that  every  word  in  the chosen  generators, using the  relations  above,  can be put in  the normal form (\ref{nformrookdn}),  and  hence  $\mathring{\RR}_{d,n}$  coincides with  $\RR_{d,n}$,  so that  Theorem \ref{Rookabacus} is proved.\end{proof}
\section{Tied  (ramified)  monoids by  deframization}

Tied  monoids were introduced in \cite{AiJuJKTR2016} and generalized in \cite{ArJuSF2021, AiArJuJPAA2023}.
  In particular,  according to  this last reference,  the  monoids  considered here can be regarded as submonoids   of
 the partition monoid, see \cite{KuMaCEJM2006}.  Thus the objects named  ``tied monoid'' in this paper  are the so called ramified monoids, see \cite[Definition 10]{AiArJuMMJ}.
  We denote by $tM$ the tied monoid associated to $M$. More precisely,  for $M= S_n$, $tM$ equals  $tS_n$ as presented in \cite[Proposition 17]{AiArJuJPAA2023}, for $M= IS_n$, $tM$ is as presented in \cite[Theorem2]{AiArJuMMJ}, for $M=Br_{n}$, $tM$ denote the so called ramified of $M$  as   defined in \cite[Definition 20]{AiArJuJPAA2023}.

We recall that the identity diagram of the partition monoid has  $n$  vertical parallel  lines.  The diagram of a generator    of $M_n$ differs  from the identity  in the  sense  that some of  the  parallel lines are  replaced  by  arcs (broken or not)  in different positions. We  say  that  these  arcs form the   {\it core } of the generator.  The core of the  identity is  the whole set of  $n$ lines: see Figure \ref{qwe}.

\begin{figure}\label{qwe}
\includegraphics[scale=1]{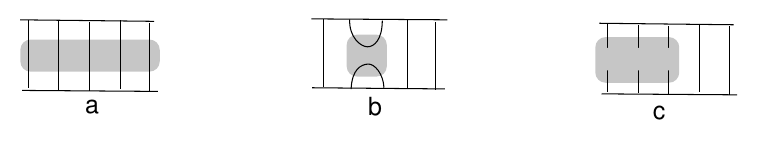}
\caption{ 
The core  of  (a) identity,  (b) of  the  tangle  generator $t_2$ and (c) of  the generator $p_3$,  shown in gray }\label{Fig19}
\end{figure}

 Diagrammatically, an element of $tM$   is  obtained  by  adding {\it  ties}  connecting pairs of arcs   of $M$.  The  ties,  represented in the diagram as  waved lines,   by  definition  define  a partition of  the set  of  generalized  arcs: two generalized  arcs connected  by  a tie belong to the  same  block of the partition.

Now,   the  tied  monoids  studied  in  \cite{ AiJuJKTR2016, AiArJuJPAA2023}, have been  provided with a  presentation:  the tied monoid  $tM$  of $M =\langle X, R\rangle$ has the   set   $X$  of generators  and relations  $R$   plus a  set $H$ of generators   that  behaves  as  ties    due to  some    new relations.
The  diagram of a  generator  in $H$  has  either a tie   connecting two parallel  lines  of  the identity or connecting the  arcs of  the  core of  a generator of $X$.

 Let  $M_{d,n}$ denote the $d$-framization of the monoid $M$, where $M$ is any monoid considered  here. The  aim of this  section is  to  recover $tM$  from    $ \CC[M_{d,n}]$.
More  precisely,  we realize the defining generators of $tM$  as certain sums in $\CC[M_{d,n} ]$, in a way similar as how the bt-algebra was constructed from the Yokonuma-Hecke algebra.

\subsection{A recipe for  deframization}

 Before  explaining   our  recipe for  deframization, we make some  observations  that  explain the  choice  of  some  elements of the monoid  algebras involved in the procedure.

  As  we  already  remarked,  in all  framed  monoids the  commutation  rules of the  framing generators with the other  generators  mean, in term of  diagrams,  that  the  beads slide  freely  along  each  arc.
 But  beads    cannot  jump  from  an  arc  to  another. In    Figure \ref{Fig14}  we  show  some   forbidden jumps between  the  parallel vertical lines and  between  arcs  of the cores of the monoids generators  here  chosen:  a)  $z_i \not= z_j$;  b)  $z_i t_i \not=t_i z_i$;  c) $z_i p_i\not= p_i z_i$.
 \begin{figure}[H]
 \includegraphics[scale=0.8]{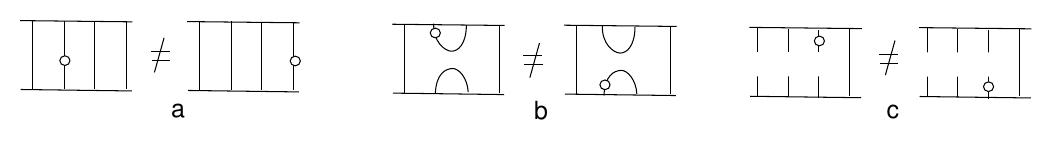}
\caption{  The forbidden jumps of  beads.}\label{Fig14}
\end{figure}
Observe  that  the  forbidden jump  $z_i s_i\not=s_i z_i$ can be written  as  $z_i s_i\not= z_{i+1} s_i$  by relation (\ref{sizi}), therefore  this case  is  included in  case  (a) with  $j=i+1$.

However,  in the monoid algebra, we can introduce    {\it bridge}  elements, i.e.  elements  that  allows  the framings to move  from an  arc to another one.

Let $g$ be a  generator of  $M_n$. We denote by  $X(g)$  the  set of  indices of the  core of  $g$.  For instance, $X(1)=\{1,\dots,n\}$, $X(s_i)=X(t_i)=\{i,i+1\}$,  $X(r_i)=\{i\}$, $X(p_i)=\{1,\dots,i\}$.
\begin{definition} Let $g$ be a generator  of  $M_n$ satisfying  in $M_{d,n}$  $z_i g\not=g z_j$ for some $i,j \in X(g)$.  The {\it bridge}    $\beta_{i,j}(g)$ is an element of  $\CC[M_{d,n}]$  satisfying  \begin{equation}\label{bridge} z_i \beta_{i,j}(g)=\beta_{i,j}(g)z_j
\end{equation}
The pair   $i,j$ of indices, possibly coinciding,  is  ordered.
\end{definition}
Observe that  in general $\beta_{i,j}(g)\not=\beta_{j,i}(g)$.
We  give here  some  examples. \begin{enumerate} 
\item[1)] Let  $g= 1\in \CC[C_{d,n}]$.  If  $i\not=j$ then $ z_i \cdot 1  \not= 1 \cdot z_j$. So, for  every pair $i,j$, $i\not=j$, we define
\begin{equation}\label{eij}
 \bar e_{i,j}:= \frac{1}{d}\sum_{k=0}^{d-1}z_i^k z_{j}^{-k}.
\end{equation}
 It satisfies  eq. (\ref{bridge}), so it is a  bridge. Moreover, it commutes with any $z_k$  and satisfies  $\bar e_{i,j}=\bar e_{j,i}$, therefore it satisfies  (cf. \cite[3.1]{JuLaTA2007})
 \begin{equation}\label{zeij} z_i \bar e_{i,j}=z_j \bar e_{i,j}= \bar e_{i,j}z_j =\bar e_{i,j}z_i.\end{equation}
Since $\CC[C_{d,n}]$ is a  submonoid  of all  framed  monoids   considered here,  $\bar e_{i,j}$ lies in  all  the  corresponding  monoid  algebras.

\item[2)]  In $\CC[J_{d,n}]$  and in $\CC[Br_{d,n}]$, consider $g=t_i$. We have  $z_j t_i  \not= t_i z_j$ only if  $j=i,i+1$.  The corresponding bridge $\beta_{i,i}(t_i)$ is given by:
\begin{equation}\label{barfi} \bar f_i   :=  \frac{1}{d}\sum_{k=0}^{d-1}z_i^k t_i  z_{i}^{-k}. \end{equation}
Indeed, it satisfies eq. (\ref{bridge}) with $i=j$. Moreover,  in virtue of  (\ref{Z2}),
in each term of the sum (\ref{barfi}), the $z_i$ at right  (or at left) can be replaced by $z_{i+1}$, so that   $\bar f_i$ results also as  the bridge $\beta_{i,j}(t_i)$ with $j=i+1$, and  satisfies
\begin{equation}\label{zfi} z_i \bar f_i=z_{i+1} \bar f_i= \bar f_iz_{i}=  \bar f_iz_{i+1}.\end{equation}

Observe also that  $\bar f_i$ commutes  with all framings $z_j$:
\begin{equation}\label{zfi2} z_j\bar f_i=\bar f_iz_{j},\quad \forall j\end{equation}

\item[3)]
In $\CC[R'_n]$, the element  $r_i$  satisfies $z_i r_i \not= r_i z_i$. The element
\begin{equation}\label{qi} \bar q_i  =  \frac{1}{d}\sum_{k=0}^{d-1}z_i^k r_i  z_{i}^{-k}\end{equation}
satisfies
\begin{equation} z_i\bar q_i = \bar q_i z_i.  \label{ziqi} \end{equation}
So, it is a bridge with $i=j$.  
Observe that also $\bar q_i$ commutes with all  framings. 
 
\item[4)]
Also, in $\CC[R'_n]$, the element $p_h$ does not commute with  all  $z_i$ such that $i\le h$. The element
\begin{equation}\label{wijh} \bar w_h^{(i,j)}  =   \frac{1}{d}\sum_{k=0}^{d-1}z_i^k p_h z_{j}^{-k}, \quad  i,j\le h,  \end{equation}
 satisfies eq. (\ref{bridge}):
$  z_i \bar w_h^{(i,j)}  = \bar w_h^{(i,j)} z_j$, 
   and is  therefore a  bridge, non symmetrical by  exchange of $i$ with $j$. If $i\not=j$, it  commutes only with all framings $z_k$ such that  $k>h$.

\end{enumerate}
In terms of  diagrams (see  Figure \ref{Fig15}), we say that
  $\bar e_{i,j}$  allows  $z_i$  to  jump from  the $i$-th  to  $j$-th line  and  viceversa; 
  that  $\bar f_i$ allows $z_i$  and  $z_{i+1}$  to jump  from a  bracket to  the  other of  the  core of $t_i$;
  that  $\bar q_i$ makes possible the jump of   $z_i$  from a broken arc   to the other of the core of the generator $r_i$, and $\bar w_h^{(i,j)}$ 
 allows $z_i$  to jump from the top broken arc at place  $i$ to the bottom one  at place $j$   of the core of  $p_h$, see  Figure \ref{Fig15}c,  where $h=3$, $i=1$ and $j=3$.

 \begin{figure}[H]
 \includegraphics[scale=0.8]{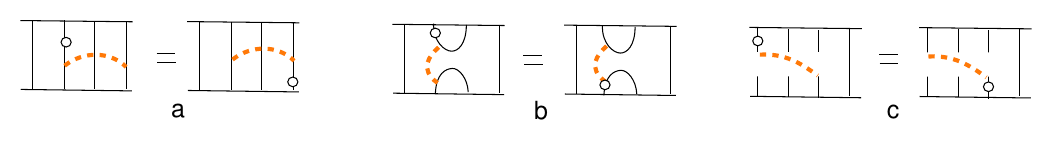}
\caption{  Examples of  bridge elements. }\label{Fig15}
\end{figure}

Still,  we  have sometimes  to  extend  the monoid  algebra, which we shall consider with coefficients in a ring $\mathbb K$.
Let   $M_{d,n}=\langle Y, S\rangle$ be one of the  framed monoids considered   above,  and $A:=\KK[M_{d,n}]$, where  $\KK=\CC$ or  $\KK= \CC[\alpha^{\pm 1}_1,\dots,\alpha^{\pm 1}_{d-1}]$ and  the $\alpha_i$ are indeterminates over $\CC$.

The extension of {scalars} in the monoid algebra is  needed when  in $M_{d,n}$ there  are  relations  meaning that  loops or  segments  with  framings have to be neglected.   In particular:
\begin{enumerate} \item[-]  in $\KK[J_{d,n}]$ and  in  $\KK[Br_{d,n}]$,  relations  (\ref{Z4}) are replaced by
               \begin{equation}\label{KKZ4}  t_i z_i^k t_i = \alpha_k t_i;
               \end{equation}
               \item[-] in $\KK[R'_{d,n}]$   relations  (\ref{Rdn2})   and  (\ref{Rdn3}) are replaced     by
               \begin{equation}\label{KKRdn2} p_i z_1^{m_1}  \cdots z_i^{m_i}  p_j = p_j z_1^{m_1}  \cdots z_i^{m_i} p_i = \alpha_{m_1}\dots \alpha_{m_i} p_j, \qquad  1\leq i\le j\leq  n ,
               \end{equation}
               \begin{equation}\label{KKRdn3}  r_i z_i^k r_i = \alpha_k r_i,
               \end{equation}
\end{enumerate}
respectively  ($\alpha_0=1$). See  Remark \ref{alfak} for  an  example explaining  the  necessity of this  extension.

 \begin{definition}\label{deframizationdef}
Build up 
a $\CC$-algebra $D(A)$  obtained from $A$   by the  following procedure, starting from a (finite) subset $B\subset A$  of bridge elements   and a finite set  $R_B$ of ``bridge relations'' (expressing the commutation relations of the bridges elements within themselves and with the generators from $Y$).
\begin{enumerate}
 \item[(i)]  Define  the set $Y^*= Y\cup B^*$,  where   $B^*$ is  a  set  in bijection with  $B$. Let
 $b\leftrightarrow b^*$ denote this bijection.
\item[(ii)]  Let  $F(Y^*)$ be the free associative algebra generated by $Y^*$.  For $r\in R_B$ let $r^*$ be  the element  of $F(Y^*)$ obtained by replacing $b\in B$ with $b^*$.
\item[(iii)] Define $\tilde A$ as the $\CC$-algebra, obtained by  specializing to 1 the parameters $\alpha_k$ in the quotient  $F(Y^*)/I$, where $I$ is the ideal generated by $S\cup \{r^*\mid r\in R_B\}$.
 \item[(iv)] Define 
 $$D(A)=\tilde A/I_z,$$
 where $I_z$ is the ideal generated by   $z_1-1, \ldots , z_n-1$.
\end{enumerate}
 Moreover,  we say  that a  (tied) monoid $tM$  is  a deframization  of    $M_{d,n}$ if  $D(A)= \CC[tM]$. 
 \end{definition}

\begin{remark}  Observe  that  the deframization  depends on the  choice of the   set $B$,  so it is not unique.
\end{remark}

 \begin{remark}  Recall  that  a  tied  monoid  $tM$ is  a ramified  monoid, in the  sense  that  the  ties of its  generators define  a  ramified partition of  $[2n]$,  see \cite[Definition 10]{AiArJuMMJ} for details.  Define  the  framization   $tM_{d,n}$ of  a  ramified (tied) monoid  $tM$ as  for  a  non ramified one,  with  new  relations expressing the property that   the tied  generators  allow the  framing to  move along each  block of  the partition.  The   algebra  $\tilde A$  turns out  therefore to  coincide with $\CC[tM_{d,n}]$,   for instance see  the  example in the    next Section  4.2.
\end{remark}


The rest of the section  provides  the  deframization of all framed monoids considered above.

\subsection{Tied Symmetric  monoid  as  deframization  of   $ S_{d,n}$}
We will show that the tied symmetric monoid $tS_n$,
 introduced  in  \cite{AiJuJKTR2016} and denoted  $TS_n$  therein, can be  viewed  as a deframization of $S_{d,n}$.
 It  is presented  by   generators $s_1,\ldots,s_{n-1}$,   $e_1,\ldots,e_{n-1}$,   and   relations  (\ref{Coxeter-Moore1})-(\ref{Coxeter-Moore2})  and:
\begin{align}
e_i^2 =e_i,&\quad e_ie_j=e_je_i,\label{TSn1}\\
s_ie_j =e_js_i&\quad\text{if }|i-j|\neq 1,\quad e_is_js_i=s_js_ie_j\quad\text{if }|i-j|=1,\label{TSn2}\\
e_ie_js_i=e_js_ie_j&=s_ie_ie_j\quad\text{if }|i-j|=1.\label{TSn3}
\end{align}

\begin{lemma}
For all  $i\in [1,n-1]$, we consider in $\CC[S_{d,n}]$  the   elements
\begin{equation}\label{ei} \bar e_i:=\bar e_{i,i+1}=\frac{1}{d}\sum_{k=0}^{d-1}z_i^k z_{i+1}^{-k}.    \end{equation}
They satisfy relations \eqref{TSn1}-\eqref{TSn3}  and 
\begin{align}\label{zei1} z_i \bar  e_i&=z_{i+1}\bar  e_i= \bar  e_i z_i= \bar e_i  z_{i+1}\\
z_j \bar e_i&=\bar e_i z_j, \quad |i-j|>1.\label{zei2}\end{align}
\end{lemma}
\begin{proof} Relations \eqref{TSn1}-\eqref{TSn3} are checked  in e.g. \cite{JuJKTR2004}.  Relation \eqref{zei1} is verified by an easy direct calculation; relation \eqref{zei2}  is obvious (indeed it holds for any $i,j$).
\end{proof}
\begin{proposition}\label{D(Sdn)} Declare $B=\{\bar e_i\mid 0<i<n\}$ to be the set of bridge elements and $R_B$, given by  \eqref{TSn1}-\eqref{TSn3}, \eqref{zei1}, \eqref{zei2}, be the set of bridge relations. Then 
$tS_n$  is  a  deframization of    $S_{d,n}$.
\end{proposition}
\begin{proof} Use the notation of Definition \ref{deframizationdef}. Note that the $\CC$-algebra $\widetilde A$  has the following presentation 
\begin{align*}\widetilde A=&\langle s_1,\ldots,s_{n-1}, z_1,\ldots,z_{n},\bar e_1^*, \ldots,\bar e_1^*|\\&\eqref{Coxeter-Moore1}-\eqref{Coxeter-Moore2}, \eqref{TSn1}^*-\eqref{TSn3}^*, \eqref{zei1}^*, \eqref{zei2}^*,\eqref{sizi}, \eqref{PreCdn}\rangle\end{align*}
so that 
\begin{align*}D(A)=&\langle s_1,\ldots,s_{n-1},\bar e_1^*, \ldots,\bar e_1^*|\eqref{Coxeter-Moore1}-\eqref{Coxeter-Moore2}, \eqref{TSn1}^*-\eqref{TSn3}^*\rangle.\end{align*}
It is now clear that the map  $z_i\mapsto z_i, \bar e^*_i\mapsto e_i$ extends to an isomorphism $D(S_{d,n})\cong \CC[tS_n]$.
 \end{proof}

\subsection{Set partition monoid  as deframization of  $C_{d,n}$}
 Recall that  the abacus monoid $C_{d,n}$ can also be regarded as a framization of the trivial group of the partition monoid. Observe  that  in $\CC[C_{d,n}]$ relations  (\ref{zeij})  hold.
 Using these  relations one can  prove  the following
\begin{lemma}\label{relbareij}
  The  elements $\bar{e}_{i,j}$ defined by \eqref{eij} satisfy:
 \begin{align}
 \bar{e}_{i,j}^2 &=\bar{e}_{i,j}\quad\text{for all }i<j,\label{bare1}\\
 \bar{e}_{i,j}\bar{e}_{r,s}&=\bar{e}_{r,s}\bar{e}_{i,j}\quad\text{for all }i<j\text{ and }r<s,\label{bare2}\\
\bar{e}_{i,j}\bar{e}_{i,k}&=\bar{e}_{i,j}\bar{e}_{j,k}=\bar{e}_{i,k}\bar{e}_{j,k}\quad\text{for all }i<j<k.\label{bare3}
\end{align}
\end{lemma}

 We can choose $B=\{\bar{e}_{i,j}\}$ as bridge elements and (\ref{bare1})-(\ref{bare3}),\eqref{zeij}  as bridge relations
The algebra  $\tilde A$,  has   generators  $z_k, \bar e_{i,j}^*$ satisfying  $\eqref{bare1}^*-\eqref{bare3}^*, \eqref{PreCdn}, \eqref{zeij}^*$, and it   coincides with $\CC[P_{d,n}]$, i.e. the  algebra of the  framed  partition monoid, see Definition \ref{Pdn}.

 Its quotient  $D(C_{d,n})$  modulo
the  ideal $I_z$ is   evidently isomorphic to  $\CC[P_n]$,
  that is,  $P_n$ is a deframization of $C_{d,n}$.

\subsection{Tied Rook monoids  by deframization}

\subsubsection{Deframization of $\RR_{d,n}$}
The  tied  Rook monoid $t\RR_n$
 is presented by generators $s_1,\ldots,s_{n-1}$,  $e_1,\ldots,e_{n-1}$, $r_1,\ldots,r_n$, $q_1,\ldots,q_n$ satisfying {\normalfont(\ref{TSn1})-(\ref{TSn3})} and  the following relations (see \cite[Theorem 3]{AiArJuJPAA2023}):
\begin{align}
 r_i^2=r_i,&  \quad
 r_ir_j=r_jr_i,\label{r1}\\
 s_ir_j&=r_{s_i(j)}s_i,\label{r2}\\
 r_is_ir_i&=r_ir_{i+1},\label{r3}\\
q_i^2&=q_i,\label{ris1bis}\\
q_iq_j&=q_jq_i\label{ris1}\\
q_ie_j&=e_jq_i,\label{ris2}\\
s_iq_j&=q_{s_i(j)}s_i,\label{ris3}\\
e_ir_je_i=e_iq_j,\quad\text{if }j=i,i+1&,\quad e_ir_j=r_je_i,\quad\text{if }j\neq i,i+1,\label{ris4}\\
r_iq_j=q_jr_i\quad\text{for all }i,j&,\quad q_ir_i=r_i,\label{ris5}\\
r_je_ir_j&=r_j,\quad\text{if }j=i,i+1,\label{ris6}\\
r_ie_ir_{i+1}&=s_iq_ir_{i+1}.\label{ris7}
\end{align}

\begin{remark}  The  cardinality   of  $t\RR_n$  is  given by
  $\sum_{k=0}^n  { \binom{n} {k}}^2 k!  Bell(2n - k) .$  These numbers are  $3, 39, 971, 38140, 2126890, 157874467, 14928602309$  for  $n=1,\dots , 7$,  respectively.
\end{remark}
We recall  that the diagram of the generator $q_i$   is the  tied  version  of  that of  $r_i$, see  Figure \ref{Fig17} b. Therefore,   in  $R'_{d,n}$  it is  natural to  take  the  bridge elements  $\bar e_i$ as in (\ref{ei}) and $\bar  q_i$  as in (\ref{qi}).

 \begin{remark}
Note that,  if  we use  the  presentation of  $R'_{d,n}$  with  the $p_i$ as generators, it is  natural to take the  bridges   elements $\bar w_i$  defined  from (\ref{wijh})   as:
\begin{equation}\label{wiii}  \bar w_i:= \bar w_i^{i,i}.
\end{equation}
The diagram of $\bar w_i$   is  shown in   Fig.  \ref{Fig17}a.   In this  case  we  get  by  deframization   the  same  tied rook monoid $t\RR_n$, presented  by  generators $s_i$ ,  $p_i$, $e_i$  and $w_i $.  The  original presentation  is  recovered  by using  eq. (\ref{RHO2})  together with
\begin{equation}  \bar w_i = r_1 r_2 \cdots  r_{i-1} \bar q_i .  \qquad  1<i\le n . \label{wi}
                \end{equation}
                \end{remark}

 \begin{figure}[H]
 \includegraphics[scale=0.8]{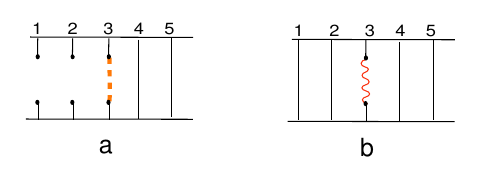}
\caption{a)  the  bridge $\bar w_3$ in  $R_{d,5}$; b) the  diagram of  $q_3$ in  $tR'_5$.  }\label{Fig17}
\end{figure}
\begin{lemma} Set $\KK=\CC[\alpha^{\pm 1}_1,\dots,\alpha^{\pm 1}_{d-1} ]$. The bridge elements $\bar e_i, \bar q_i$ verify, in $\KK[R'_{d,n}]$,  relations \eqref{r1}-\eqref{r3}, \eqref{ris1}-\eqref{ris5}, \eqref{ris7}, \eqref{zei1},\eqref{zei2},  and 
\begin{align}
\bar q_i^2&=\bar  q_iZ_i,\label{R1}\\
r_i \bar e_i r_i &= r_i Z_{i+1}, \label{R2}\\
r_{i+1} \bar e_i r_{i+1}&= Z_i  r_{i+1},\label{R3}\\
 z_j \bar q_i &= \bar q_i z_j,\end{align}
where 
\begin{equation}\label{Zi}  Z_i:= \frac{1}{d}\sum_{k=0}^{d-1} \alpha_k z_i^{-k}. \end{equation}
\end{lemma}
\begin{proof} 


Equations (\ref{r1})-(\ref{r3}) are  satisfied in   $R'_n$ and  hence in   $t\RR_n$ and  $R_{d,n}$, since both  $t\RR_n$ and  $R'_{d,n}$ contain $R_n$ as  submonoid.
Relation \eqref{R1} reads:
\begin{align*}\bar q_i^2 & = \frac{1}{d^2}\sum_{k,h=0}^{d-1} z_i^k r_i z_i^{-k} z_i^h r_i z_i^{-h}\stackrel{\eqref{KKRdn3}}{=}
  \frac{1}{d^2}\sum_{k,h=0}^{d-1} z_i^k \alpha_{h-k}  r_i z_i^{-k} z_i^{k-h} \\  &=\left( \frac{1}{d }\sum_{k=0 }^{d-1} z_i^k   r_i z_i^{-k}  \right) \frac{1}{d }\sum_{j=0}^{d-1} \alpha_j z_i^{-j} = \bar q_i Z_i.
\end{align*}
Relation \eqref{ris1}  trivially  from  (\ref{fRdn3}), as well as
eq. (\ref{ris2})  if $i \not= j,j+1$.

Eq. (\ref{ris3}) gives:
$$ s_i \bar q_j  = \frac{1}{d }\sum_{k =0}^{d-1} s_i z_j^k r_j z_j^{-k}  \stackrel{\eqref{M2rook2},\eqref{M2rook3} }{=}
  \frac{1}{d }\sum_{k 0}^{d-1} z_{s_i(j)}^k   r_{s_i(j)} z_{s_i(j)}^{-k} s_i  = \bar q_{s_i(j)} s_i.
$$
The  first  equation in (\ref{ris4})  and  (\ref{ris2}) with $j=i$ reads:
\begin{align*} \bar e_i r_i \bar e_i &=  \frac{1}{d^2 }\sum_{k,h =0}^{d-1} z_i^k  z_{i+1}^{-k} r_i  z_i^h  z_{i+1}^{-h}=\frac{1}{d^2 }\sum_{k,h }^{d-1} z_{i+1}^{-k} z_{i+1}^{-h}   z_i^k r_i z_i^{-k} z_i^{h+k}  =\\
 &=   \frac{1}{d  }\sum_{ h }^{d-1} z_{i+1}^{-k-h} \left( \frac{1}{d  }\sum_{k=0}^{d-1}   z_i^k r_i z_i^{-k} \right)z_i^{h+k}=\left( \frac{1}{d  }\sum_{k=0}^{d-1}   z_i^k r_i z_i^{-k} \right)\frac{1}{d  }\sum_{ h+k=0 }^{d-1} z_{i+1}^{-k-h}z_i^{h+k}= \bar q_i \bar e_i \\
 & \stackrel{\eqref{ziqi}}{=} \left( \frac{1}{d  }\sum_{k+h=0}^{d-1}z_{i+1}^{-k-h}z_i^{h+k}   \right)\frac{1}{d  }\sum_{  k=0 }^{d-1} z_i^k r_i z_i^{-k}= \bar e_i \bar q_i.
 \end{align*}
 The case  $j=i+1$ is  analogous.
The  first  equation in (\ref{ris5}), for  $i\not=j$, follows  from  (\ref{fRdn3}) and (\ref{r1}),  while  the second one follows   from (\ref{fRdn4}).
Let's check \eqref{R2}, \eqref{R3}:
$$r_i  \bar e_i r_i =  r_i  \frac{1}{d  }\sum_{k=0 }^{d-1} z_i^k  z_{i+1}^{-k}  r_i= \frac{1}{d  }\sum_{k=0 }^{d-1} r_i  z_i^k r_i z_{i+1}^{-k}  \stackrel{\eqref{KKRdn3}}{=}\frac{1}{d  } r_i \sum_{k=0 }^{d-1} \alpha_k  z_{i+1}^{-k}= r_i Z_{i+1}.   $$

 $$r_{i+1}  \bar e_i r_{i+1} =  r_{i+1} \frac{1}{d  }\sum_{k=0 }^{d-1}   z_i^k  z_{i+1}^{-k}  r_{i+1}= \frac{1}{d  }\sum_{k=0 }^{d-1}  z_i^k r_{i+1} z_{i+1}^{-k} r_{i+1}  \stackrel{\eqref{KKRdn3}}{=}\frac{1}{d  }\sum_{k=0 }^{d-1}    z_i ^{ k}\alpha_k r_{i+1} =  Z_i  r_{i+1}.  $$
Let's finally check (\ref{ris7}):
 \begin{align*}r_i \bar e_i r_{i+1}  &=  r_i \frac{1}{d }\sum_{k =0}^{d-1} z_i^k  z_{i+1}^{-k} r_{i+1}  = \frac{1}{d}\sum_{k=0}^{d-1}  z_{i+1}^{-k} r_i  z_i^k r_{i+1}    =\frac{1}{d}\sum_{k=0}^{d-1}  z_{i+1}^{-k} s_i s_i r_i  z_i^k r_{i+1 }    =\\
 &=  \frac{s_i}{d}\sum_{k=0}^{d-1}  z_{i}^{-k}   s_i r_i  z_i^k r_{i+1 }  =  \frac{s_i}{d}\sum_{k=0}^{d-1}  z_{i}^{-k}   s_i (r_i   r_{i+1 } ) z_i^k   \stackrel{\eqref{Rdn3}}{=} \frac{s_i}{d}\sum_{k=0}^{d-1}  z_{i}^{-k}   s_i (r_i  s_i r_{i } ) z_i^k =\\
 &=\frac{s_i}{d}\sum_{k=0}^{d-1}  z_{i}^{-k} r_{i+1}  s_i s_i r_{i }  z_i^k=\frac{s_i}{d}\sum_{k=0}^{d-1}  z_{i}^{-k} r_{i+1}   r_{i }  z_i^k=\frac{s_i}{d}\sum_{k=0}^{d-1} \big( z_{i}^{-k} r_{i }  z_i^k \big) r_{i+1} =s_i \bar q_i  r_{i+1}.
\end{align*}
 The remaining relations are easily verified.
\end{proof}
Arguing as in the previous subsections, and noting that $Z_i=1$ in $D(\KK[R'_{d,n}])$, we see that $\CC[t\RR_n ]$ is a deframization of $R'_{d,n}$.
\begin{remark}\label{alfak} Consider  the  first relation we have verified in the  above proof.    If  we do not  take  the  extension $\KK$, i.e. relation (\ref{Rdn3}) is not  replaced by  (\ref{KKRdn3}),    we  get  $\bar q_i^2   = \bar q_i Z_i= Z_i \bar q_i = Z_i r_i Z_i$,  where   $Z_i= \frac{1}{d}\sum_{k=0}^{d-1} z_i^k$. By applying  the  procedure  we  get $q_i^2= q_i=r_i$  which are not relations   of $tR'_n$.
\end{remark}

\subsubsection{Deframization of $R_{d,n}$}  Next  we  get  another  version  of the  tied rook  monoid, that  we denote  $tR_n$.
\begin{definition}  The tied monoid  $tR_n$  is  defined  by  generators  $s_i, p_i, e_i$  subject to  the  defining relations of  $R_n$  and  to  the  relations:
   \begin{align}    e_i p_j  &= p_j  e_i= p_j,  \qquad  i\le j ,\label{tRn1}  \\
                 e_i p_j  &= p_j   e_i  \qquad  i> j .\label{tRn2}
\end{align}
\end{definition}
Firstly,  we observe that  in $\KK[R_{d,n}]$  the  bridge  elements $\bar q_i$  and  $\bar w_i$  do  not exist; namely, by (\ref{fRdn4})  they satisfy
$$    \bar q_i =  \bar w_i = p_i.$$

So,  let the set $B$  contain  only the  bridges $\bar{e_i}$   (\ref{ei}),  hence we may take  $\KK=\CC$. Moreover,    the 
bridge relations   are  (\ref{TSn1})-(\ref{TSn3}),  \eqref{zei1}, \eqref{zei2}  and
\begin{align}  \bar e_i p_j  &= p_j \bar e_i= p_j,  \qquad  i<j ,\label{DRn1}  \\
               \bar e_i p_i &= Z_{i+1} p_i , \qquad   p_i \bar e_i  = p_i Z_{i+1} ,\label{DRn2}\\
               \bar e_i p_j  &= p_j \bar e_i  \qquad  i>j .\label{DRn3}
\end{align}
They  can be   verified using  (\ref{fRdn3})  and (\ref{fRdn4}). Hence  $ tR_n$ is a deframization of $R_{d,n}$.

Using  the  same method  used in proving  \cite[Theorem 2]{AiArJuMMJ}  for  $t\RR_n$, i.e.  the  ramified  inverse  symmetric  monoid,  we can prove the following
\begin{proposition} The monoid  $tR_n$  coincides  with the submonoid of the  ramified  monoid of $ R_n$ \cite[Definition 10]{AiArJuMMJ}  consisting of  double  partitions $(I\preceq J)$  where  $I$  is  an element  of $R_n $,  and     $J$  contains the  same  singletons  as  $I$   and  blocks that  are unions of  lines of  $I$.
\end{proposition}

The cardinality  of  $tR_n$  is given by  $\sum_{k=0}^n  { \binom{n} {k}}^2 k!  Bell( k)$.  These numbers,  for  $n=1,\dots,7$   are  2, 9, 76, 1001, 1866, 464737, 14813324.

\subsection{Tied Jones monoid  as  deframization of $J_{d,n}$}
The tied Jones monoid $tJ_n$,  is  defined by  generators $t_1,\ldots , t_{n-1}$, $e_1, \ldots ,e_{n-1}$,
$f_1,\ldots ,f_{n-1}$ satisfying the defining relations of $J_n$, (\ref{TSn1}) and the relations:
\begin{align}
f_i^2&=f_i,\\
f_if_j &= f_j f_i \quad \text{for $|i-j|>1$,}\label{MtTL1}\\
e_i t_i = &t_ie_i =t_i, \quad f_ie_i = f_i,\label{MtTL2} \\
e_if_j = &f_j e_i,\quad t_if_i = f_it_i= t_i, \label{MtTL3}\\
 t_ie_j = &e_jt_i,\quad t_i f_j = f_jt_i, \quad \text{for all $|i-j|>1$,}\label{MtTL4}\\
t_ie_jt_i=& t_i, \quad\text{and}\quad f_i e_j = e_jt_ie_j, \quad \text{for all $|i-j|=1.$\label{MtTL5}}
\end{align}
This monoid was studied in  \cite[Subsection 5.3]{AiArJuMMJ}, where it is called planar  ramified Jones monoid. It was also proven that  its cardinality  equals the Fuss-Catalan number $C_{4,1}(n)= \frac{1}{4n+1}\binom{4n+1}{n}$.
The monoid $tJ_n$ can  be also  obtained  as a specialization of the two-parameter algebra $t\TL_n(x,y)$ introduced below in
Definition \ref{tTLn}.

Set $\KK=\CC[\alpha_1^{\pm 1},\dots,\alpha_{d-1}^{\pm 1}]$.
 In this case we take as bridges the set formed by the  elements $\bar{e}_i$ and $\bar{f}_i$ defined in \eqref{ei}, \eqref{barfi}, respectively. 
The bridge relations are \eqref{MtTL1}-\eqref{MtTL5}, \eqref{zei1}, \eqref{zei2}, \eqref{zfi}, \eqref{zfi2} and 
\begin{equation}f_i^2=f_iZ_i.\label{MtTL1bis}\end{equation}

 Only  (\ref{MtTL1bis}) 
 will be verified, since the others follow  easily. We have
\begin{eqnarray*}
(\bar{f}_i)^2& =&\frac{1}{d^2}\sum_{r,s}z_i^{r} t_i z_{i}^{-r}
z_i^{s} t_i z_{i}^{-s} = \frac{1}{d^2}\sum_{r,s}z_i^{r} t_i
z_i^{s-r} t_i z_{i}^{-s}\stackrel{a=s-r}{=}\frac{1}{d^2}\sum_{r,a}z_i^{r} t_i
z_i^{a} t_i z_{i}^{-r-a} \\
&\stackrel{\eqref{KKZ4}}{=}&\frac{1}{d^2}\sum_{r,a}z_i^{r} \alpha_a  t_i z_{i}^{-r-a} =\frac{1}{d^2}\sum_{r,a}z_i^{r}   t_i z_{i}^{-r} \alpha_a z_i^{-a} =
\bar{f}_iZ_i .
\end{eqnarray*}



\subsection{Tied Brauer monoid as  deframization of  $Br_{d,n}$}\label{tBr}
 The tied  Brauer monoid  $tBr_n$  was introduced  in \cite{AiArJuJPAA2023} as the ramified Brauer monoid, and it was  proven that it has  a presentation with generators $s_1, \ldots , s_{n-1}$, $t_1, \ldots , t_{n-1}$, $ e_1, \ldots , e_{n-1}$, and
$f_1, \ldots , f_{n-1}$ satisfying the defining relations of $Br_{n}$, $tS_n$ and $tJ_n$, together with the following relations.
\begin{align}
& f_is_j = s_j f_i\quad  for\ |i-j|>1,\label{fisj} \\
& f_i s_i =  s_if_i = f_i, \label{fisi}\\
& s_if_js_i = s_jf_is_j  \quad   for\ |i-j|=1.\label{sifjsi}
\end{align}
\begin{remark}
The relation $f_if_jf_i = e_jf_ie_j$ with $|i-j|=1$,   appears in the original definition of $tBr_n$, but  it is a consequence of the other relations, cf. \cite[(i)Lemma 2]{AiJuMathZ2018}, hence it is omitted.\end{remark}

Set $\KK=\CC[\alpha^{\pm 1}_1\dots,\alpha^{\pm 1}_{d-1}]$. We show that $tBr_n$ is a deframization of $Br_{d,n}$. We have bridge elements
$\bar f_i, \bar e_i$ and, by the analysis developed in the previous cases the only non obvious bridge relations to verify are (\ref{fisj})-(\ref{sifjsi}).

Relation (\ref{fisj}) follows directly from (\ref{sizi}) and (\ref{B02}).
For relation (\ref{fisi}), we have:
$$\bar{f}_is_i=\frac{1}{d}\sum_{k=0}^{d-1}z_i^{k} t_i z_{i+1}^{-k} s_i= \frac{1}{d}\sum_{k=0}^{d-1}z_i^{k} t_i s_i z_{i}^{-k} =
\frac{1}{d}\sum_{k=0}^{d-1}z_i^{k} t_i   z_{i}^{-k}=  \bar{f_i}.$$
Analogously, $s_i\bar{f}_i = \bar{f}_i$.

Finally, we deal with relation (\ref{sifjsi}).   We have: \begin{equation}\label{prev} s_i\bar{f}_js_i=\frac{1}{d} \sum_{k=0}^n s_iz_j^kt_jz_j^{-k}s_i, \end{equation} and   each summand on  the  right hand of \eqref{prev} side equals
 $$ z_i^ks_it_js_iz_i^{-k}
\stackrel{\eqref{B04}}{=}  z_i^ks_jt_is_jz_i^{-k}
= s_jz_j^kt_iz_j^{-k}s_j =s_jz_i^kt_iz_i^{-k}s_j,
$$
so that  by taking the sum  we get  $s_j \bar f_i s_j$.

 \begin{remark}  The  cardinality of $tBr_n$ is $(2n-1)!! Bell(n)$. \end{remark}

\section{Framization and Deframization  of  Algebras}

In this  section  we  discuss how  to  extend the definition of  framization and  deframization  to   certain algebras, which  are  strictly  related to the  monoids  we have  considered  before.
We will work in the following framework. Let $A$ be an algebra with a fixed presentation $\langle X,R\rangle$, with both $X$ and $R$ finite.
 Let  $\aaa=(a_1,\ldots a_n)$ be a  finite set  of parameters. We will consider the class $\mathcal F_A$ of $\CC[\aaa^{\pm 1}]$-algebras   $A(\aaa)$ having a presentation $\langle X,R(\aaa)\rangle$, where the relations $R(\aaa)$ depend polinomially on $\aaa$ and are such that $R(1)=R$ (where $R(1)$ means setting $a_i=1$ in each relation). We will write 
 \begin{tikzcd}
{A(\aaa)} \arrow{r}{\aaa\to 1}  & A.\end{tikzcd}


Let  $M_{d,n}$ be the $d$-framization of  $M$, where $M$ denotes any monoid or group considered  here, and $A$ the corresponding monoid algebra. Let $B$ be the chosen  set of bridge elements.
 Recall  that  every  element  $b^*\in B^*$ is  in  bijection  with a  bridge element  of  $\KK[M_{d,n}]$. Since  the image  of a  bridge  element in the deframization   is  either 1  or  a   generator   $x$ of $M$, the map $\gamma: b^*\mapsto x$  is  well  defined.   Observe   that $\gamma$ maps  $b^*$ to 1  when $b^*$ is  a  tie, and  to $x$  when $b^*$ is   the  tied  version of the generator $x$. Therefore $\gamma$ can be  defined  independently from the framization and indeed it maps $tM$  to $M$.

 In the  diagrams  below  the  horizontal  arrows correspond  to the   specializations to 1 of the algebra  parameters.  Moreover,  we assume that   $\aaa  \subseteq \aaa'$  and  $\aaa \subseteq \aaa''$.
 By $\aaa' \to \aaa$  we  mean   that    $\CC[(\aaa')^{\pm 1}]$ is  sent to $\CC[\aaa^{\pm 1}]$.


\begin{equation}\label{d1}
\begin{tikzcd}
{\mathcal A(\aaa)} \arrow{r}{\aaa\to 1}  & \CC[M] \arrow[d, "fram"]\\
     f {\mathcal A(\aaa')}\arrow[u,"\aaa'\to\aaa" near start,  "z_i\to 1" near end] \arrow{d}[swap]{defram}\arrow[r,"\aaa'\to 1 "]
& \CC[M_{d,n}]\arrow[d,"defram"]\\ t{\mathcal A}(\aaa'') \arrow[r,"\aaa''\to 1 "]
    & \CC[tM]
   \end{tikzcd}
   \end{equation}
\begin{equation}\label{d2}
\begin{tikzcd}
  t{\mathcal A}(\aaa'') \arrow[r,"\aaa''\to 1 "] \arrow[d, "\gamma" near start, "\aaa''\to\aaa" near end]
    & \CC[tM] \arrow[d, "\gamma"] \\
  {\mathcal A}(\aaa) \arrow[r,"\aaa\to 1 "]
& \CC[M] \end{tikzcd}
\end{equation}

\begin{definition} \label{framdefram} We say that

 \begin{enumerate}

\item the   algebra  $\mathcal A(\aaa)\in\mathcal F_{\CC[M]}$  admits  a $d$-framization     if there  exists an algebra    $f\mathcal A(\aaa')\in \mathcal F_{\CC[M_{d,n}]}$  that  makes  the upper diagram in \eqref{d1}  commutative.  The algebra $f\mathcal A$  is  said  to be {\sl framed}.

\item  the deframization of  the framed  algebra  $f\mathcal A$  is an algebra    $t\mathcal A(\aaa'')\in\mathcal F_{\CC[tM]}$  that  makes  the  lower diagram  in \eqref{d1} commutative.  The algebra $t\mathcal A$ is  said  to be {\sl tied}.

\item  The algebra  $\mathcal A$ admits  a tied version   if there  exists an algebra    $t\mathcal A(\aaa'')\in\mathcal F_{\CC[tM]}$  that  makes  the  diagram \eqref{d2} commutative.
 \end{enumerate}

 \end{definition}

\begin{remark}\label{deframalgebra}  Note that, if a  framed  algebra  $f \mathcal A$  exists, then  the procedure  of deframization explained in definition  4.1  is well defined  also for  $f \mathcal A$, starting from it instead  of $\CC[M_{d,n}]$. Therefore, an algebra  admitting  framization  admits   a tied  version,  but  not viceversa,  as  we will  see.
\end{remark}

   In the  next  table  we show  the   monoids we have considered  with the    framed  and  tied  version   and the related choice of parameters (see Definition \ref{framdefram}).  The  corresponding algebras do  not always  admit   a  framization (e.g.,  $\BMW_n$  and  $\mathcal R_n$  if the corresponding framed monoid is $\RR_{d,n}$). However,  in the  case of $Br_n$, the  algebra $\BMW_n$  admits  a  tied  version.
\[
\begin{array}{c|c|c||c|c|c||c|c|c}
 M_n &    \A            &  \aaa & M_{d,n}  &   f\A    & \aaa' & tM_n & t\A  & \aaa''     \\  \hline
 S_n &    H_n           &  v    & S_{d,n}   & \mathrm Y_{d,n}   &v     &  tS_n   &  \mathcal E_n    &    v  \\  \hline
 J_n&    \TL_n          &  x   & J_{d,n}  &  \TL_{d,n} & x,y_i  & tJ_n &  t\TL_n  &  x,y     \\  \hline
  R_n  & \mathcal R_n  &  v    &  R_{d,n} & \mathcal R_{d,n} & v  &  tR_n  & t\mathcal R_n  &v \\ \hline
   R_n  & \mathcal R_n  &  v    & \RR_{d,n} & - & -  &  t\RR_n  & - & - \\ \hline
 Br_n &  \BMW_n   &  a,q   & Br_{d,n} & -   & - &  tBr_n  & t\BMW & a,q,x
\end{array}
 \]

   In the    first  example (Section \ref{sectionYH}) we see how   well-known algebras,  the  Y-H-algebra and the  bt-algebra,  can  be  considered  as framization  and  deframization of  the Hecke  algebra.

In the  second example  (Section \ref{framedTL}) we define  the  framization and  deframization of  the Temperley-Lieb  algebra, getting  a  framed  and  tied  version of   $\TL_n$.

 In the   remaining  sections, we give  the  presentation of the  framed/tied  version of  the rook  algebra  and  of the  tied version of the BMW  algebra.

\begin{remark} The fact that the dimension  of  a  framed (or tied) algebra  coincides  with  the  cardinality of the corresponding  framed (tied)  monoid  is  proved only  for the  first  two  examples.
\end{remark}

 \subsection{The Y-H  algebra  and  the bt-algebra  }\label{sectionYH}
 \subsubsection{The Y-H  algebra}
 The Yokonuma-Hecke algebra  $\mathrm{Y}_{d,n}(v)$  has its origin  as the centralizer of the permutation representation  of the  general lineal group over  a finite field with respect to a maximal unipotent subgroup, see \cite{YoCRASP1967}.
 The algebra $\mathrm{Y}_{d,n}(v)$  has a presentation  with    invertible generators
 $g_1,\ldots ,g_{n-1}$, $z_1,\ldots ,z_n$ satisfying the following relations:
\begin{align}
\label{yk1}
 z_i z_j  = & z_j z_i,\quad
  z_i^d=  1,\\
\label{yk2}
z_j g_i  =  & g_i z_{s_i(j)}, \\
\label{yk3}
 g_i g_j  =  g_j g_i\quad \text{for $\vert i - j\vert > 1$}, &\quad\text{and}
\quad  g_i g_j g_i =  g_j g_ig_j\quad \text{for $\vert i - j\vert = 1$},\\
\label{yk4}
g_i^2 = & 1  +  (v-v^{-1}) \bar{e}_ig_i,
\end{align}
where $\bar{e_i}$ is defined by formula \eqref{ei}.

\begin{remark}\label{11}
The presentation of Y-H algebra used here is not the original presentation relating it to knot theory, see \cite{JuJKTR2004}. For more details on the different presentations of the Yokonuma-Hecke algebra we refer the reader to, e.g.  \cite{MaIMRN2017, DaDoART2017}.
\end{remark}
\begin{remark}
Similarly  to the Hecke algebra, the Yokonuma-Hecke algebra can be thought as a deformation of the group algebra  $\mathcal{S}_{d,n}$.
Set $\KK=\CC(v)$.  Remark \ref{RemPre} says that  $\mathrm{Y}_{d,n}(v)$ can be also defined as the quotient of $\KK[\mathcal{F}_{d,n}]$ modulo  the two-sided ideal generated by the following {\it framed quadratic} expressions:
$$
\sigma_i^2 -1-(v-v^{-1})\bar{e}_i\sigma_i, \quad  i\in \ldbrack 1,n-1\rdbrack.
$$
\end{remark}
Since the Hecke algebra is a  flat deformation of $ \CC[S_n]$  and the Yokonuma-Hecke algebra is a  flat deformation of  $\CC[\mathcal{S}_{d,n} ]$  (see Remark \ref{PreCdn}),   according to  definition \ref{framdefram} we say that the Yokonuma-Hecke is a framization of the Hecke algebra.

\subsubsection{The bt-algebra }
The bt-algebra  $\mathcal{E}_n(u)$ was  defined  in \cite{AiJu}. It has been originally defined  through a presentation with {\it braid} generators $g_1,\ldots, g_{n-1}$ and {\it ties}  generators $e_1,\ldots, e_{n-1}$ satisfying
the following relations:
\begin{align}
\label{bt1}
 g_i g_j  =  g_j g_i\quad \text{for $\vert i - j\vert > 1$}, &\quad\text{and}
\quad  g_i g_j g_i =  g_j g_ig_j\quad \text{for $\vert i - j\vert = 1$},\\
\label{bt2}
e_ie_j =  & e_j e_i,  \quad  e_i^2  =   e_i,\\
\label{bt3}
e_i g_j   = & g_j e_i,\quad  \text{if $| i  -  j|\not=1$},\\
\label{bt4}
e_ig_jg_i = g_jg_ie_j \quad \text{and}\quad &
e_ie_jg_i =  e_j g_i e_j  = g_ie_ie_j \quad \text{if $\vert i  -  j\vert =1$},\\
\label{bt5}
g_i^2  = & 1  +  (v-v^{-1})e_ig_i.
\end{align}

 We say that the bt-algebra is a deframization of the Yokonuma-Hecke algebra      since  can be obtained  by  the procedure   in definition \ref{deframizationdef},  as outlined in  Remark \ref{deframalgebra}.  Its construction was done in fact by removing  the framing generators of  the Yokonuma-Hecke algebra and by  considering a  new abstract generator $e_i$  satisfying the  same  relations  as the $\bar{e}_i$   in the Yokonuma-Hecke algebra.
Note that  $\mathcal{E}_n(1)=\CC[S_{d,n}]$  and   the bt-algebra    has indeed  dimension $n!Bell(n)$.

\subsection{Framed  and tied  Temperley-Lieb algebra}\label{framedTL}

\subsubsection{The algebra  $\TL_{d,n}$}

Let  $n$ and $ d$ be two  positive integers.  Let $x,y_1,\ldots , y_{d-1}$ be  parameters and put $\KK = \CC(x,y_1,\ldots , y_{d-1})$ . We denote by $\mathrm{TL}_{d,n}$ the   $\KK$-algebra generated by $t_1, t_2,\ldots, t_{n-1}$, $z_1, \ldots  ,z_n$ satisfying relations  \eqref{J02}, \eqref{PreCdn}, \eqref{Z2}, \eqref{Z3} and the following relation, where  $y_0=y_d=1$:
\begin{align}
  t_iz_i^k t_i  =  & x y_k t_i \quad  k\in   \ZZ/d\mathbb{Z}.\label{AZ4}
\end{align}
\begin{remark}
Observe  that $\mathrm{TL}_{1,n}$ corresponds to the classical Temperley-Lieb algebra  $\mathrm{TL}_n(u)$, with $u=x $. Also note that taking $u=x$  and $y_1=\cdots =y_{d-1}=1$  in the above relations we have an algebra epimorphism  $\mathrm{TL}_{d,n}\to \mathrm{TL}_n(u)$, defined by  $t_i\mapsto t_i$ and $z_i\mapsto 1$.
\end{remark}
The main result of this section is the following
\begin{theorem}\label{Main}
The algebra  $\mathrm{TL}_{d,n}$  is  a  d-framization of $ \mathrm{TL}_n$. Its dimension is  therefore
  $Cat_n d^n$.
\end{theorem}
Its proof  follows  from Proposition \ref{generatorstilde} below and its Corollary.

\subsubsection{The  diagram algebra $\widetilde{\mathrm{TL}}_{d,n}$}
We  denote by    $\widetilde{\mathrm{TL}}_{d,n}$ the  vector space over  $\KK$  with basis  $\widetilde{J}_{d,n}$. Define a product on this  basis  as follows. Given  $a,b\in\widetilde{J}_{d,n}$  with $ab=c$, we define $\widetilde{\mathrm{TL}}_{d,n}$ \begin{equation}\label{productab} ab= x^{n_1+\ldots +n_d} y_1^{n_1}y_2^{n_2}\cdots y_d^{n_d}c,\end{equation} where $n_p$ is the number of loops   with $p$ beads that are neglected  by the product. We extend this product to $\widetilde{\mathrm{TL}}_{d,n}$ by bilinearity.
\begin{proposition} Formula \eqref{productab} endows $\widetilde{\mathrm{TL}}_{d,n}$ with a $\mathbb K$-algebra structure.
\end{proposition}
\begin{proof} We have only to verify that the product  \eqref{productab} is associative. This
follows  from the  fact that  the  concatenation product is
associative, also if loops  are not  neglected.  Indeed,  given $a,b,c$
in $\widetilde{J}_{d,n}$
consider the concatenation $a\circ  b \circ  c$.   It contains  loops of
different type:
ab-loops, that  are  formed  in the concatenation $a \circ  b$,  bc-loops,
 that  are formed  in the concatenation $b \circ  c,$  and  abc-loops
that  are  formed  only in the  double  concatenation.
Now,  if  we  take  $(a \circ  b) \circ  c$,  the coefficient  firstly
changes  by  the  contribution of  ab-loops, and  afterwards  changes for
the  contribution  of  the bc-loops  and  abc-loops.  If  we take $a
\circ  (b \circ c)$ the coefficient  firstly   change  by  the
contribution  of  the bc-loops,  and  afterwards  by  the contribution of
the  ab-loops  and  abc-loops.  The final  result is  the  same.
\end{proof}

\begin{proposition}\label{generatorstilde}
 $\widetilde{\mathrm{TL}}_{d,n}$   is  the algebra presented   by  the  generators  $\T_i$ and  $\OO_i$ of the  Abacus monoid, subject  to the same relations \eqref{O002}-\eqref{O3}     whereas  \eqref{O002} nd relation \eqref{O4}  are  replaced by
\begin{equation}\label{ykx}  \T_i \OO_i^k \T_i =   x\ y_k  \T_i, \end{equation}
where  $k\in \ZZ/d\ZZ$ and  $y_0=y_d=1$.
\end{proposition}

\begin{proof}  Every  element   $\tilde w$ of  $\widetilde{\mathrm{TL}}_{d,n}$  is a linear combination with coefficients in $\KK$  of  elements $w_j$ of  $\mathring J_{d,n}$. We   prove that  every $w_j$  can be  written,    using the relations of  the  algebra,
  in the normal  form \eqref{dotnormalform}  and  that, in  so doing, its  coefficient changes in agreement  with (\ref{productab}). We will follow  step by step the reduction procedure  shown in  Lemma \ref{Lem02}. The coefficient of this  element  during this  procedure  changes  by $x\,y_k$ every time  we  replace  by $\T_i$ the product  $\T_i \OO_i^k \T_i$  or $\T_i \OO_{i+1}^k \T_i$,  according to  relation (\ref{ykx}).
Observe  that, in terms  of  diagrams, a  product of  generators $w_j$ may contain any  number of  loops,  each one  involving  any  number  of  generators  $\T_i$  and  $z_i$.   During  the reduction,  each  loop  diminishes  its  length in terms of generators   by  substituting  the  expression   $\T_i \T_{i\pm 1} \T_i $  with  $\T_i$.   The  beads  are  collected  together  till  the loop  is  reduced to the minimal  loop
 in the  expression    $\T_i \OO_{i}^k\T_i$ or  $\T_i \OO_{i+1}^k \T_i$.  So,  each  loop  disappears only by (\ref{ykx}) when all  beads  are  collected  together, and  therefore  it  gives  the   contribution to  the  coefficient  according  to (\ref{productab}).
In  figure  \ref{f6} we illustrate this  procedure in a case with $n=6, d=11$.
\begin{figure}[H]
\includegraphics[scale=0.55]{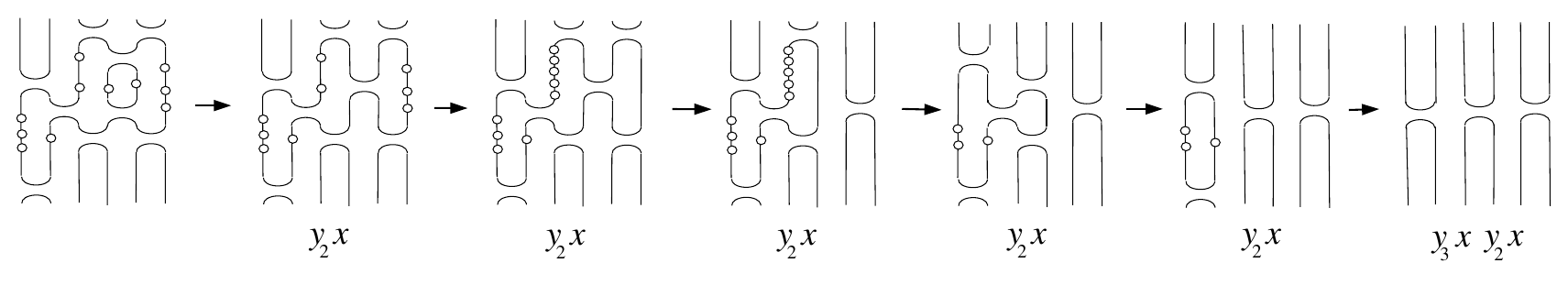}
\caption { Procedure. }\label{f6}
\end{figure}

 Note  that the multiplicative contributions   are non-zero, and, also, that the coefficient remains unchanged   in all  other  relations. This  concludes the proof.
\end{proof}

\begin{corollary}
The map  $z_i \mapsto \OO_i$, $t_i \mapsto  \T_i$ extends to a $\mathbb K$-algebra isomophism $\mathrm{TL}_{d,n}\to\widetilde{\mathrm{TL}}_{d,n}$.
 \end{corollary}

\subsubsection{   Tied Temperley-Lieb algebra}

\begin{definition}\label{tTLn}  Let  $x$ and  $y$  be two  indeterminates  over $\CC$. The tied Temperley-Lieb algebra $t\TL_n $ is the $\CC(x,y)$-algebra  generated  by $t_1,\ldots, t_{n-1}$,  $e_1,\ldots, e_{n-1}$, $f_1,\ldots, f_{n-1}$, all  commuting  each other  when the  difference of the indices is  different  from $1$,  and  satisfying
\begin{itemize}
\item  for $|i-j|=1$
\begin{align} 
&e_ie_j=e_je_i,\\
 &  t_i t_j t_i    =   t_i,\label{titjti}\\
& t_ie_jt_i=y t_i, \label{tiejti} \\
& f_i e_j  =  e_j f_i   =  e_j t_i e_j, \label{ejtiej}  \end{align}
\item for every $i$
 \begin{align}
 & t_i^ 2 =  x t_i, \label{ti2}\\
 & e_i^2  =   e_i  \label{ei2} \\
 &  f_i^2= y f_i, \label{fi2}\\
& t_i e_i =     t_i,   \label{tiei}\\
 & f_i e_i  =   f_i, \label{fiei}\\
& f_i t_i =y t_i.  \label{fiti}
\end{align}
\end{itemize}
 \end{definition}
\begin{proposition} The algebra $t\TL_n $ is obtained  as   deframization of  $\TL_{d,n}$.
\end{proposition}
  \begin{proof}
  The  proof  is very similar to  that given in the monoid case,  according to  Remark  \ref{deframalgebra}.
  Observe  that,  by the   lower diagram in \eqref{d1},  $t\TL_n(1,1)=\CC[J_n]$, hence  $\dim t\TL_n =|tJ_n|$.   Looking at   diagram \eqref{d2},  we see that $\gamma(t\TL_n(x,x))=\TL_n(x)$.
\end{proof}

\subsection{Framed  and  tied  rook  algebras}

\subsubsection{The framed rook algebra  $\mathcal R_{d,n}$}
Firstly, observe that  the  rook algebra $\mathcal R_n(v)$ (Section \ref{rA})   is  a  deformation of  the monoid algebra  $\CC[R_n]$ (Section \ref{RM}), in the  sense  that  the  relation  $s_i^2=1$ is  deformed in    $T_i^2=1+(v-v^{-1})T_i$,  and relations (\ref{Mrook3}) and (\ref{Mrook4}) are deformed   respectively in   (\ref{Arook3})  and (\ref{Arook4}),  and  $\dim  \mathcal R_n(v)=|R_n|$  (see \cite{HaJA2004}).\par
The  framed  rook algebra  $\mathcal R_{d,n}$, see  \cite{ChEaJA2023} can be obtained  as   deformation of  the framed monoid algebra $\CC[R_{d,n}]$,  see Section \ref{Rdn}.  
More precisely, the braid   generators $T_1,\dots,T_{n-1}$ satisfy  the same  relations  as  $\mathcal R_n(v)$, but  the  quadratic  relation   is  replaced by the quadratic  equation  of the  YH-algebra:   $T_i^2=1+(v-v^{-1})\bar e_i T_i$. The  other relations  of  $\mathcal R_n(v)$  remain unaltered  as  well as the  relations  of $R_n$ involving  framing  generators $z_i$,  written in terms of   $P_i$  and  $T_i$  instead of $p_i$ and $s_i$  respectively.
Finally, observe  that  the YH-algebra is  a  subalgebra  of  $\mathcal R_{d,n}(v)$.

\subsubsection{The tied rook algebra  $t \mathcal R_{n}$}

By  applying the procedure of  deframization to  $\mathcal R_{d,n}$,  we  get  the tied  algebra $t \mathcal R_{n}$,  which  can be considered  as  a  deformation of  the  tied  rook  monoid algebra $\CC[t R_n]$.
It is  generated  by  braid generators  $T_1\dots,T_{n-1}$, $P_1,\dots, P_n$, and ties generators $e_1,\dots,e_{n-1}$. Its relations are  those   of  the  bt-algebra,    relations  (\ref{Arook1}--(\ref{Arook4})  and  relations (\ref{tRn1}),(\ref{tRn2}), with  $p_i$  replaced by $P_i$.
Note that the bt-algebra is  a  subalgebra  of  $t\mathcal R_n(v)$.

\subsection{Tied BMW algebra}

In  \cite{AiJuMathZ2018}  the  first  two  authors have introduced  the  tied  BMW  algebra,    as  the knot  algebra   generalizing  the BMW  algebra in order  to  recover, by  a  trace, a  Kauffman type  invariant of  tied  links,  that was also  defined in \cite{AiJuMathZ2018}.

We  recall  here  the definition of the  tied  BMW  algebra, by  using  the presentation    of  the $\BMW$  algebra  chosen in section \ref{BMWalgebra}.

 Let $a,q,x$ be three indeterminates over $\CC$.
 
 \begin{definition}\label{deftBMW} The   $t\BMW$  algebra   is the algebra   over $\CC[a,q,x]$   generated by
$g_1^{\pm 1},\ldots ,g_{n-1}^{\pm 1}$ (braids),
$t_1,\ldots ,t_{n-1}$ (tangles), $e_1,\ldots ,e_{n-1}$ (ties), $f_1,\ldots ,f_{n-1}$ (tied tangles),
 all commuting each other  when the  difference of the indices is  different from 1,  satisfying 
 \begin{itemize}
 \item relations \eqref{bt1}--\eqref{bt4}, 
 \item  for  $|i-j|=1$, relations    \eqref{titjti}--\eqref{ejtiej}   and
 \begin{align}
 &t_i g_jt_i   =  at_i,  \label{tBMW1}  \\
   &  g_i g_j t_i    =   t_j g_i g_j = t_jt_i  \label{tBMW2}  \\
  & g_i t_j g_i  = g_j^{-1}t_i g_j^{-1},\quad g_if_j g_i  =  g_j^{-1} f_i g_j^{-1}, \label{tBMW3}\\
 & g_i t_j t_i  =  g_j^{-1}t_i,  \quad
 t_i t_j g_i    =   t_i g_j^{-1},
 \label{tBMW4}
 \end{align}
 \item 
 for  every $i$, relations \eqref{ti2}, \eqref{tiei}, \eqref{fiei} and 
\begin{align}
& g_i t_i =  a^{-1} t_i, \quad   f_i g_i     =  a^{-1} f_i,  \label{tBMW5}\\
 & g_i-g_i^{-1}   =  (q-q^{-1})(e_i - f_i).\label{tBMW6}
\end{align}
\end{itemize}
\end{definition}

Observe  that  \begin{equation}\label{fisquare}f_i^2= yf_i,\end{equation} where   $y=\frac{a-a^{-1}}{q-q^{-1}} + 1$, and that $ f_i t_i= t_i f_i =  y  t_i $.
Moreover,  just as the BMW algebra  is a deformation of the Brauer monoid, the tied BMW algebra is a deformation of tied Brauer monoid  algebra (Section \ref{tBr}), in particular  $t\BMW(1,1,1)_n=\CC[tBr_n]$.  Also, by diagram \eqref{d2}, the algebra $\BMW(a,q)_n$ is  recovered  from $t\BMW(a,q,x)_n$  by the map $\gamma$  sending $e_i$ to  1 , $f_i$ to $t_i$,  and  $x$ to $y(a,q)$, cf.   (\ref{parx}).

 \begin{remark}The tied BMW algebra cannot be obtained  via  deframization  since a  framed BMW algebra does not exist in the  sense of Definition 5.1.   Indeed,  in such a  framed BMW algebra,  
the equation  (\ref{tBMW6}), with $e_i$  and  $f_i$  replaced  respectively by  $\bar e_i$ and  $\bar f_i$, should  hold.
However, Relation (\ref{fisquare}),  satisfied by   $\bar  f_i$, should  imply    $Z_i=y$ for every $i$,  see  eq.(\ref{MtTL1bis}), and hence entailing the non-existence  of the  central horizontal  arrow  in diagram \eqref{d1}
 \end{remark}

\begin{remark}\label{Loic}  Several  framizations of  algebras  have  been  defined by using  a purely algebraic  approach, see \cite{JuLaSKE, FlJuLaJPAA2018, ChPo2017, GoJuKoLa2017, LaPoarXiv2023}.     It is worth noting that  the  framed  TL  and BMW  algebra   defined in the above references are  not    deformations of the corresponding  framed   monoid  algebras (rather,  they are  constructed as quotients). However, they do  not have an immediate  interpretation  in terms  of  diagrams. Recently, in \cite{LaPoarXiv2023},  the  tied version of  the  BMW algebra,   called {\it BMW algebras of braids and ties} there,  was introduced. It turns out to be  a deformation of  a monoid  algebra.  This  monoid,   introduced in \cite[Section 6]{AiArJuJPAA2023} and called   $\mathrm{bBr}_n$, is in fact    a   submonoid  of  the  tied Brauer  monoid,  in particular it is obtained  by  $tBr_n$  by  sending to 1  the  tangle  generators.
\end{remark}

{\bf Acknowledgments.} The author wish to thank the referee, whose suggestions have greatly improved the exposition. They also thank Domenico Fiorenza and Claudio Procesi for useful discussions, and James Mitchell for highlighting a typo in a preprint version. The second author  acknowledges the support of the CRM-Pisa through the research project in pairs,
where this work began, as well as the support of the Project {\it Teoria delle rappresentazioni e applicazioni} that allowed him a nice stay at {\it  Sapienza Universit\`a di Roma} in 2022.

\section*{Declarations}
\noindent{\bf Competing Interests.} The authors have no competing interests to declare that are relevant to the content of this article.

\noindent {\bf Data Availability Statement.} Data sharing not applicable to this article as no datasets were generated or analysed during the current study.

\noindent{\bf Funding.} Jesus Juyumaya was partially
supported by Fondecyt No. 1210011. Paolo Papi has been partially supported by Progetto di Ateneo 2023, {\it Representation Theory and Applications}, Sapienza Universit\`a di Roma.

\end{document}